\documentclass[oneside,11pt]{article}

\usepackage{amsthm}
\usepackage{amsmath}
\usepackage{amssymb}
\usepackage{pxfonts}
\usepackage{dsfont}
\usepackage{graphicx}
\usepackage{eucal}
\usepackage{mathrsfs}
\usepackage{pifont}
 \usepackage{sectsty}
\usepackage{amscd}
\usepackage{color}
\usepackage{fancyhdr}
\usepackage{framed}
\usepackage[all]{xy}
\usepackage[backref=page]{hyperref}
\usepackage{hyperref}
\usepackage{cancel}
\usepackage{enumitem}
\usepackage[normalem]{ulem}

\definecolor{shadecolor}{rgb}{0.8,0.8,0.8}

\newtheorem{theorem}{Theorem}[section]
\newtheorem{definition}[theorem]{Definition}
\newtheorem{proposition}[theorem]{Proposition}
\newtheorem{lemma}[theorem]{Lemma}

\newcommand{\brk}[1]{\left(#1\right)}          % \brk{.}     => (.)
\newcommand{\Brk}[1]{\left[#1\right]}          % \Brk{.}     => [.]
\newcommand{\BRK}[1]{\left\{#1\right\}}        % \BRK{.}     => {.}
\newcommand{\Abs}[1]{\left| #1 \right|}        % \Abs{.}     => |.|
\newcommand{\Norm}[1]{\left\| #1 \right\|}     % \Norm{.}    => ||.||
\newcommand{\pd}[2]{\frac{\partial#1}{\partial#2}}

\newcommand{\secref}[1]{Section~\ref{#1}}
\newcommand{\thmref}[1]{Theorem~\ref{#1}}

\newcommand{\propref}[1]{Proposition~\ref{#1}}
\newcommand{\lemref}[1]{Lemma~\ref{#1}}

\newcommand{\beq}{\begin{equation}}
\newcommand{\eeq}{\end{equation}}

\providecommand{\e}{\varepsilon}

\providecommand{\R}{\mathbb R}

\newcommand{\Textand}{\qquad\text{ and }\qquad}

\newcommand{\dist}{{\operatorname{dist}}}
\newcommand{\Hom}{{\operatorname{Hom}}}
\newcommand{\SO}{{\operatorname{SO}}}
\newcommand{\Vol}{{\operatorname{Vol}}}
\providecommand{\rank}{{\operatorname{rank}}}

\newcommand{\limn}{\lim_{n\to\infty}}

\newcommand{\g}{\mathfrak{g}}
\newcommand{\h}{\mathfrak{h}}
\renewcommand{\b}{\mathfrak{b}}

\newcommand{\frakX}{\mathfrak{X}}

\newcommand{\hn}{\mathfrak{n}}
\newcommand{\euc}{\mathfrak{e}}
\newcommand{\M}{{\mathcal M}}
\newcommand{\N}{{\mathcal N}}
\newcommand{\E}{{\mathcal E}}
\newcommand{\calC}{{\mathcal C}}
\newcommand{\bbN}{{\mathbb N}}
\newcommand{\Volume}{\textup{dVol}_\g}
\renewcommand{\Vol}{\textup{dVol}}

\newcommand{\tM}{\tilde{\M}}
\newcommand{\tF}{\tilde{F}}
\newcommand{\tg}{\tilde{\g}}
\newcommand{\tf}{\tilde{f}}

\newcommand{\Rd}{\R^{d+1}}

\newcommand{\Ppar}{P_\parallel}

\newcommand{\RomanNumeralCaps}[1]
{\MakeUppercase{\romannumeral #1}}

\renewcommand{\SO}{\operatorname{SO}}
\renewcommand{\O}{\operatorname{O}}
\newcommand{\II}{\RomanNumeralCaps{2}}
\newcommand{\Imm}{\operatorname{Imm}}

\makeatletter
\newcommand{\extp}{\@ifnextchar^\@extp{\@extp^{\,}}}
\def\@extp^#1{\mathop{\bigwedge\nolimits^{\!#1}}}
\makeatother

\numberwithin{equation}{section}

\begin{document}
	\makeatletter
%	\def\thanks#1{\protected@xdef\@thanks{\@thanks
%			\protect\footnotetext{#1}}}
%	\makeatother
	
	\title{Asymptotic rigidity for shells in non-Euclidean elasticity
%		\thanks{\,\, RK was partially supported by the Israel Science Foundation (Grant No.~1035/17). CM was partially supported by the Israel Science Foundation (Grant No.~1269/19). }}% \textbf{Conflict of Interest: The authors declare that they have no conflict of interest.}}
		}

	\author{Itai Alpern\footnote{Institute of Mathematics, The Hebrew University
	\newline RK was partially supported by the Israel Science Foundation (Grant No.~1035/17). CM was partially supported by the Israel Science Foundation (Grant No.~1269/19). }
			%{ \newline  Itai Alpern 
			 %\newline  Institute of Mathematics, The Hebrew University
			%	\newline  E-mail: {itai.alpern@mail.huji.ac.il}}
		 \and Raz Kupferman\footnotemark[1] %\thanks{ \newline  Raz Kupferman 
		 	%\newline  Institute of Mathematics, The Hebrew University
		 	%\newline  E-mail: {raz@math.huji.ac.il}}
		   \and   Cy Maor\footnotemark[1] %\thanks{ \newline  Cy Maor 
		   		%\newline  Institute of Mathematics, The Hebrew University
		   		%\newline  E-mail: {cy.maor@mail.huji.ac.il}}
				}
	\date{}
	
	\maketitle
	
	\begin{abstract}
		
		We consider a prototypical ``stretching plus bending" functional of an elastic shell.
		The shell is modeled as a $d$-dimensional Riemannian manifold endowed, in addition to the metric, with a reference second fundamental form. The shell is 
		immersed into a $(d+1)$-dimensional ambient space, and 
		the elastic energy accounts for deviations of the induced metric and second fundamental forms from their reference values.
		Under the assumption that the ambient space is of constant sectional curvature, we prove that any sequence of immersions of asymptotically vanishing energy converges to an isometric immersion of the shell into ambient space, having the reference second fundamental form.
		In particular, if the ambient space is Euclidean space, then the reference metric and second fundamental form satisfy the Gauss-Codazzi-Mainardi compatibility conditions.
		This theorem can be viewed as a (manifold-valued) co-dimension 1 analog of Reshetnyak's asymptotic rigidity theorem.
		It also relates to recent results on the continuity of surfaces with respect to their fundamental forms.
		
		
	\end{abstract}

	%\tableofcontents
	
	%%%%%%%%%%%%%%%%%%%%%%%%%%%%%%%%%%%%%%%%%%%%%%%%%
	\section{Introduction}
	\label{intro}
	
	The classical result, often referred to as the \emph{fundamental theorem of surface theory}, asserts that given a $ d $-dimensional, simply-connected Riemannian manifold $ (\M,\g) $ and  a symmetric tensor field $\b \in\Gamma(T^*\M\otimes T^*\M)$, there exists a smooth isometric immersion $ f:(\M,\g) \to \Rd  $ such that $ \g $ and $ \b $ are its first and second fundamental forms, if and only if $ \g $ and $ \b $ satisfy the Gauss-Codazzi-Mainardi (GCM) compatibility equations. In addition, $ f $ is unique up to isometries of $ \Rd $, i.e., the surface is uniquely determined by its fundamental forms, up to rigid motions (the two-dimensional version of this theorem appears in almost every book on differential geometry; the $d$-dimensional case can be found, e.g., in \cite{Ten71}).
	
	Shell theory is concerned with the elastic theory of thin bodies. A thin elastic body is modeled as a two-dimensional surface $\M$ endowed with a Riemannian metric $\g$ (often referred to as a reference first fundamental form) and a reference second fundamental form $\b$. A configuration of a thin body is an embedding of $\M$ into three-dimensional Euclidean space $\R^3$. 
	To each configuration $f$ corresponds an elastic energy penalizing for metric deformations; it comprises of a stretching term, accounting for in-plane deformations, and a bending term, accounting for out-of-plane buckling. A prototypical energy is 
	\beq
	\label{eq:energy_euc}
	\E_p(f) = \int_\M \brk {\kappa_s\, \dist^p_{\g,\euc}(df,O(\g,\euc)) + \kappa_b |df \circ(S-S_f)|_{\g,\euc}}^{p} \Volume.
	\eeq

	The first term in the integrand is the stretching energy density, where $ \dist_{\g,\euc}(df,O(\g,\euc)) $ is the pointwise distance of $df$ from the set of orthogonal transformations $(T\M,\g)\to \R^3$. The second term in the integrand is the bending energy density. 
	Here $S$, defined by $(S(X),Y)_\g = \b(X,Y)$, is viewed as a reference shape operator, and $S_f$ is the shape operator of $f$ in $\R^3$, defined by $(S_f(X),Y)_{f^*\euc} = \b_f(X,Y)$, where $\b_f$ is the second fundamental form of $f$ in $\R^3$ (see \secref{subsec:Second fundamental form}); $(\cdot,\cdot)_\g$ and $(\cdot,\cdot)_{f^*\euc}$ are the inner-products corresponding to the metric $\g$ and the pullback metric $f^*\euc$, with respect to which the norm $|\cdot|_{\g,\euc}$ and the distance $\dist_{\g,\euc}$ are evaluated (see Section~\ref{sec:gen_not}).
	The parameters $\kappa_s$ and $\kappa_b$ are the stretching and bending moduli, which in this work are immaterial, and can be taken equal to one.
	A similar energy which appears often in the physics literature (see, e.g., \cite{KES07,DHS11,GV11,GSD16}), is
	\beq
	\label{eq:phys_energy}
	\E(f) = \int_\M \brk{\kappa_s\, |\g - f^*\euc|^p + \kappa_b |\b-\b_f|^p} \Volume,
	\eeq
	where $ f^*\euc$ is the pullback metric (the Cauchy-Green tensor).
	Both \eqref{eq:phys_energy}  and \eqref{eq:energy_euc} can be derived as small-strain formal asymptotic limits of thin finite elasticity (they differ by a higher-order term), so in this sense they are equivalent theories; we discuss the applicability of our main theorem to the energy \eqref{eq:phys_energy} in \secref{sec:Discussion}.
	
	The fundamental theorem of surfaces interconnects with shell theory in the following way: if $f:\M\to\R^{3}$ is a smooth immersion having zero elastic energy (a \emph{reference configuration}, in the terminology of materials science), then the first fundamental form of $f$ in $\R^{3}$ must equal $\g$ whereas $S_f$ must equal $S$, meaning that the second fundamental form $\b_f$ of $f$ in $\R^{3}$ equals $\b$. 
	If the reference forms $\g$ and $\b$, which are intrinsic properties of the body, do not satisfy the GCM equations, then there is no such immersion $f$, i.e., there do not exist  smooth strain-free reference configurations. 
	Such a scenario is typical to so-called \emph{pre-stressed materials}, leading to the concept of \emph{incompatible shell theory} \cite{ESK09a,KS14}.
	
	Thus, if $\g$ and $\b$ are compatible, then there exists a zero-energy reference configuration.
	As shown in \cite[Cor.~5.4]{BLS16}, \cite[Lem.~3.1]{MS19}, the converse is true as well, i.e., if there exists a configuration $f$ (which does not need to be, a-priori, smooth), such that $\E(f)=0$, then $\g$ and $\b$ are compatible.
	
	As in most non-convex variational problems, it is a priori hard to determine whether $ \E $ admits a minimizer.
	It is therefore natural to ask whether $\inf\E=0$ implies that $\g$ and $\b$ are compatible, and in this case, whether
	any minimizing sequence converges (modulo a subsequence) to a zero-energy reference configuration.
	This equivalence, namely that the incompatibility of $\g$ and $\b$ (a geometric incompatibility) will result in some finite elastic energy (an elastic incompatibility), is assumed implicitly in non-Euclidean elasticity and incompatible shell theory.
	In this paper, we prove this equivalence in a more general setting.
	
	The natural space to study elastic energies of the type \eqref{eq:energy_euc} is the space of Sobolev immersions, given by the following definition:
	\begin{definition}\label{def:imm_p_Euc}
		Let $(\M,\g)$ be a compact, connected, oriented, $d$-dimensional Riemannian manifold with Lipschitz-continuous boundary.
		For $p\in[1,\infty)$, the \textbf{space of $p$-Sobolev immersions} is given by
		\[
		\Imm_p(\M;\Rd) :=\BRK{f\in W^{1,p}(\M;\Rd) ~:~ \rank \, {df}= d \,\, \text{a.e.}, \; \hn_f \in W^{1,p}(\M;\Rd)},
		\]
		where $ \hn_f $ is the unit normal vector field to $ f $ (see \secref{subsec:sobolev_immersions}).
		\\
		We say that a sequence $f_n\in \Imm_p(\M;\Rd)$ converges to an immersion $f$ in $\Imm_p(\M;\Rd)$ if $ f_n \to f $ in $ W^{1,p}(\M;\Rd) $ and  $ \hn_{f_n} \to \hn_f $ in $ W^{1,p}(\M;\Rd) $.
	\end{definition}
	
	In this paper we prove the following theorem:
	%%%%%%%%%%%%%
	\begin{theorem}
		\label{thm:main_euc}
		Let $(\M,\g)$ be as in Definition~\ref{def:imm_p_Euc}.
		Let $\b \in\Gamma(T^*\M\otimes T^*\M)$ be a symmetric tensor field and let $ 1 \leq p < \infty $. 
		Consider the energy functional $ \E_p:\Imm_p(\M,\Rd) \to \R $ defined by
		\[
		\E_p(f) = \int_\M \brk{\dist^p_{\g,\euc}(df,O(\g,\euc)) + |df \circ(S-S_f)|_{\g,\euc}^{p}} \Volume \;,
		\]
		where $ \dist_{\g,\euc}(df,O(\g,\euc)) $ is the pointwise distance of $df$ from the set of orthogonal transformations $(T\M,\g) \to (\Rd,\euc)$; the operator $S: T\M \to T\M$ is the self-adjoint linear map associated with $ \b $, i.e., $(S(u),v)_{\g} = \b(u,v)$; and $ S_f $ is the shape operator associated with $ f $.
		
		\noindent Suppose that there exists a sequence $f_n\in \Imm_p(\M;\Rd)$ satisfying 
		\[
		\limn \E_p(f_n) \to 0.
		\] 
		Then, there exists a subsequence of $f_n$ converging (modulo translations) in $ \Imm_p(\M;\Rd) $ to a smooth isometric immersion $f:\M \rightarrow \Rd$ (i.e., $f^*\euc = \g$), such that $ \b $ is the second fundamental form of $ f $ in $\Rd$.
		\\
		In particular, $\g$ and $\b$ satisfy the GCM compatibility equations.
	\end{theorem}
	%%%%%%%%%%%%%
	
	We further generalize \thmref{thm:main_euc} by changing the target manifold from $ \Rd $ into $(\N,\h)$, a compact, connected, oriented, $(d+1)$-dimensional Riemannian manifold (with no boundary) of constant sectional curvature $ \kappa $.
	To this end, we first define the space of Sobolev immersions between manifolds:
	
	\begin{definition}\label{def:imm_p_N}
		Let $(\M,\g)$ be a compact, connected, oriented, $d$-dimensional Riemannian manifold with Lipschitz-continuous boundary,
		and let $(\N,\h)$ be a connected, oriented, $(d+1)$-dimensional Riemannian manifold (with no boundary).
		For $p\in [1,\infty)$, the \textbf{space of $\N$-valued $p$-Sobolev immersions} is given by
		\[
		\Imm_p(\M;\N):=\BRK{f\in W^{1,p}(\M;\N) ~:~ \rank \, {df}= d \,\, \text{a.e.}, \; \hn_f \in W^{1,p}(\M;T\N)}.
		\]
		where $ \hn_f $ is the unit normal vector field to $ f $.
		
		\noindent We say that a sequence $f_n\in \Imm_p(\M;\N)$ converges to an immersion $f$ in $\Imm_p(\M;\N)$ if $ f_n \to f $ in $ W^{1,p}(\M;\N) $ and  $ \hn_{f_n} \to \hn_f $ in $ W^{1,p}(\M;T\N) $.
	\end{definition}
	
	%%%%%%%%%%%%%
	\begin{theorem}
		\label{THM:MAIN_CONST_CURV}
		Let $(\M,\g)$ and $(\N,\h)$ be as in Definition~\ref{def:imm_p_N}, and assume that $\N$ has constant sectional curvature $ \kappa $.
		Let $\b \in\Gamma(T^*\M\otimes T^*\M)$ be a symmetric tensor field and let $ 1 \leq p < \infty $. 
		Consider the energy functional $ \E_p:\Imm_p(\M;\N) \to \R $ defined by
		\beq\label{eq:E_pMN}
		\E_p(f) = \int_\M \brk{\dist^p_{\g,\h}(df,O(\g,\h)) + |df \circ(S-S_f)|_{\g,\h}^{p}} \Volume \;,
		\eeq
		where $ \dist_{\g,\h}(df,O(\g,\h)) $ is the pointwise distance of $df$ from the set of orthogonal transformations $(T\M,\g) \to (T\N,\h)$.
		
		\noindent Suppose that there exists a sequence $f_n \in \Imm_p(\M;\N)$ satisfying 
		\[
		\limn \E_p(f_n) \to 0.
		\] 
		Then there exists a subsequence of $f_n$ converging in $\Imm_p(\M;\N)$ to a smooth isometric immersion $f:(\M,\g) \to (\N,\h)$, such that $ \b $ is the second fundamental form of $ f $ in $ \N $.
	\end{theorem}
	%%%%%%%%%%%%%
	
	The generalization to a non-flat ambient space is not only a mathematical extension. Even without resorting to general relativity, the immersion of elastic bodies into non-flat target manifolds has attracted much interest in recent years in the physics community; see e.g., \cite{AKMMS16}.
	
	Before describing the sketch of the proof and relating our results to the surface rigidity literature, we make the following observation:
	Shell theory is a limit of ``bulk" elasticity (the elastic body and the ambient space are of the same dimension), where the elastic body has one slender dimension. A central theme in the theory of elasticity is the derivation of dimensionally-reduced theories, such as plate, shell and rod theories, as limits of bulk elasticity, in which the elastic body is a $3 $-dimensional Riemannian manifolds $(\tM,G)$ embedded in a Euclidean space of the same dimension. For $F:\tM\to \R^3$,  the elastic energy accounts for metric distortions, a prototypical energy being
	\beq
	E_{p}(F) = \int_{\tM} \dist^{p}_{G,\euc}(dF,\SO(G,\euc))\, \Vol_G,
	\label{eq:E3}
	\eeq  
	where $\SO(G,\euc)$ is the bundle of orientation-preserving isometries $(T\tM,G)\to\R^3$. Note that there is a meaning to orientation-preservation only when the source and the target are of the same dimension.
	
	Liouville's theorem, in its weak form, asserts that $E_{p}(F)=0$ if and only if $F$ is a smooth  local isometry of $(\tM,G)$ in $\R^3$, and in particular, $G$ has zero curvature. The asymptotic generalization of Liouville's theorem is Reshetnyak's celebrated rigidity theorem \cite{Res67}:
	
	%%%%%%%%
	\begin{quote}
		\emph{
			Let $\Omega\subset \R^d$ be an open, connected bounded domain, and let $1\le p<\infty$.
			If $F_n\in W^{1,p}(\Omega;\R^d)$ satisfy
			\[
			\lim_{n\to \infty} \int_\Omega \dist^p(dF_n, \SO(d))\,dx = 0,
			\]
			then $F_n$ has a subsequence converging (modulo translations) in the strong $W^{1,p}(\Omega;\R^d)$ topology to an affine mapping.
		}
	\end{quote}
	%%%%%%%
	
	In \cite[Thm.~3]{KMS19}, Reshetnyak's asymptotic rigidity theorem was generalized to a Riemannian setting:
	
	\begin{quote}
		\emph{
			Let $(\M,\g)$ and $(\N,\h)$ be compact, oriented, $d$-dimensional Riemannian manifolds with Lipschitz boundary. Let $1\le p<\infty$ and let $F_n\in W^{1,p}(\M;\N)$ be a sequence of mappings satisfying  
			\[
			\dist_{\g,\h}(d F_n, \SO(\g,\h))\to 0 \qquad\text{in $L^p(\M)$}.
			\]
			Then, $\M$ can be immersed isometrically into $\N$, and there exists a subsequence of $F_n$ converging in $W^{1,p}(\M;\N)$ to a smooth local isometry $F:\M\to\N$.  
		}
	\end{quote}
	%%%%%%%%%%
	
	The exact statement of \cite[Thm.~3]{KMS19} refers to a $C^1$ boundary, but it holds (with the same proof) to Lipschitz boundary as well.
	The same theorem holds also if $(\N,\h)$ is replaced with $\R^d$, in which case the convergence is modulo translations \cite[Cor.~7]{KMS19} (see also \cite[Thm~2.2]{LP11}).
	We note that a local isometry $F:\M\to\N$ may be considered as a rigid mapping, as it is uniquely determined by $ F(p) $ and $ dF_p $ at a single point $ p\in \M $. In this context, \thmref{thm:main_euc} can be viewed as a co-dimension $ 1 $ variant of Reshetnyak's theorem, where $ {\dist^p_{\g,\euc}(df,O(\g,\euc)) + |df \circ(S-S_f)|_{\g,\euc}^{p}} $ is a measure of the local distortion of $f$. 
	Similarly, \thmref{THM:MAIN_CONST_CURV} can be viewed as a co-dimension $ 1 $ variant of \cite[Thm.~3]{KMS19}; 
	in fact, \cite[Thm.~3]{KMS19} provides a main tool in proving \thmref{THM:MAIN_CONST_CURV}.
	Unlike \cite[Thm.~3]{KMS19}, \thmref{THM:MAIN_CONST_CURV} is currently limited to target manifolds of constant sectional curvature. 
	
	\paragraph{\textbf{Sketch of proofs}}
	We present a sketch of the proof of \thmref{THM:MAIN_CONST_CURV}, emphasizing its main ideas; the proof of \thmref{thm:main_euc} follows a similar line.
	We thicken $\M$ into a $ (d+1) $-dimensional manifold $ \tM=\M\times [-\e,\e] $. 
	We endow $\tM$ with a metric $G$, satisfying that $\M$ is isometric to $ \M\times\{0\} $ and $ \b $ is the second fundamental form of $ \M\simeq \M\times\{0\}$ in $\tM$. 
	We then extend each $ f_n\in \Imm_p(\M;\N)$ into a map $ F_n\in W^{1,p}(\tM;\N) $ and prove that $\dist_{G,\h}(dF_n,SO(G,\h))\to 0$ in $L^p(\tM) $.
	It follows from {\cite[Thm.~3]{KMS19}} that there exists a subsequence of $F_n$ converging in $ W^{1,p}(\tM;\N) $ to a smooth local isometry $ F:(\tM,G)\to (\N,\h) $.
	
	A priori, since $ \M\times\{0\} $ is of measure zero in $\tM$, the convergence of $F_n$ to $F$  in $ W^{1,p}(\tM;\N)$ does not necessarily imply the convergence of $ F_n|_{\M\times\{0\}}=f_n $ in $ W^{1,p}(\M;\N) $ to $ F|_{\M\times\{0\}}:=f $. 
	However, due to the choice of the metric $G$ and the specific way in which $f_n$ were extended to $F_n$,  we show that the sequence $f_n$ does converge to $f$ in $ W^{1,p}(\M;\N) $. 
	In addition, $ f:\M \to \N $ is an isometric immersion and $ \b $ is the second fundamental form associated with it. 
	Lastly, from the convergence of $ f_n $ to $f$ in $ W^{1,p}(\M;\N) $ and the fact that $ \limn \E_p(f_n) = 0 $ we deduce the convergence of $ \hn_{f_n} \to \hn_f $ in $ W^{1,p}(\M;T\N) $.

	\paragraph{\textbf{Related work} }
	The fundamental theorem of surface theory asserts that a surface in Euclidean space is uniquely determined (up to isometries) by its first and second fundamental forms. A natural question is in which topologies the relation between compatible fundamental forms $ \g,\b $ and the corresponding isometric immersions $ f $ is continuous. 
	
	Such questions have been analyzed thoroughly in recent years, primarily by Ciarlet and collaborators \cite{Cia03,CM19,CMM19b,CMM19}.
	Among these results, the closest to our setting is \cite[Theorem~4.2--4.3]{CMM19b}:\footnote{The statement below is not an exact quote of \cite[Theorem~4.2--4.3]{CMM19b}, but a variation that is more comparable with our main results.}
	\begin{quote}
		\emph{
			Let $ \omega $ be an open subset of $ \R^2 $, let $p>1$ and let $f\in C^1(\bar{\omega};\R^3)$ be an immersion.
			Let $\e>0$, and let $\tilde{f} \in \Imm_{2p}(\omega;\R^3)$, whose first and second fundamental forms are bounded almost everywhere by $\e^{-1}$.
			Then, there exists a constant $C=C(f,\e)$ such that, modulo a rigid transformation of $\tilde{f}$,
			\[
			\| \tilde{f} - f\|_{W^{1,p}} + \|\hn_{\tilde{f}} - \hn_f\|_{W^{1,p}} \le C\brk{\|f^*\euc-\tilde{f}^*\euc\|_{L^p} + \|\b_f-\b_{\tilde{f}}\|_{L^p}}.
			\]
		}
	\end{quote}
	Similar results in stronger topologies were proved in \cite{Cia03,CM19,CMM19b,CMM19}; these results are farther away from ours, as they involve estimates on the derivative of the induced metrics, which are not present in the elastic models that motivate this work. %typical elastic models. 
	
	The result mentioned above bears some similarities with \thmref{thm:main_euc}; we mention below some of the fundamental differences:
	
	\begin{enumerate}[]
		\item 
		\textbf{Compatibility of $ \g $ and $ \b $:}
		In \cite{CMM19b}, the reference forms $ \g $ and $ \b $ are \emph{assumed} to be compatible, whereas in \thmref{thm:main_euc} this is deduced.
		In particular, the question of the existence of a zero energy configuration for given forms of an elastic shell cannot be treated in the framework of \cite{CMM19b}.
		
		\item \textbf{Quantitative vs.\ qualitative estimates:}
		While our results only imply the convergence of configurations with asymptotically zero elastic energy, the results of \cite{CMM19b} are supplemented with a convergence rate;
		this is related to the previous point, as the key rigidity estimate in \cite{CMM19b} is the Friesecke-James-M\"uller geometric rigidity estimate, which is known only between Euclidean domains.
		Since we do not pre-assume the compatibility of $ \g $ and $ \b $, we have to rely on the Reshetnyak-like asymptotic-rigidity estimate \cite{KMS19}.
		
		\item
		\textbf{A priori assumed uniform estimates:}
		The results of \cite{CMM19b} assume pointwise bounds on the fundamental forms, whereas Theorem~\ref{thm:main_euc} does not; from the point of view of applications such a priori bounds are not typically available.
	\end{enumerate}
	
	We also note that although some aspects of the proofs in this paper are similar to their counterparts in the aforementioned literature (e.g., the proof in \cite{Cia03} also involves the thickening of the domain), the assumptions of \thmref{thm:main_euc} require different tools, mainly the use of the non-Euclidean versions of Reshetnyak's rigidity theorem, as mentioned above.

	\paragraph{\textbf{Structure of the paper}} 
	In \secref{sec:prelim}, we present some preliminary notations and results, in particular regarding the second fundamental form, Sobolev spaces between manifolds and Sobolev immersions.
	In \secref{sec:proof_main_const_curv}, we present a detailed proof of \thmref{THM:MAIN_CONST_CURV}.
	In \secref{sec:proof_main_euc}, we show how the proof of \thmref{THM:MAIN_CONST_CURV} can be adapted (and simplified) to proving \thmref{thm:main_euc}.
	In \secref{sec:Discussion}, we discus some aspects of the results which are better understood after reading the proofs. In particular, we explain (i) why in \thmref{THM:MAIN_CONST_CURV} we limit ourselves to a target manifold of constant sectional curvature; and (ii) why we took $ |df\circ(S-S_f)|_{\g,\h} $ as the bending term in our energy functional and why under physically-reasonable assumptions, it is equivalent to the more common bending term $ |\b - \b_f|_{\g,\g} $.
	We also discuss in Section~\ref{sec:Discussion} some open questions.

	%%%%%%%%%%%%%%%%%%%%%%%%%%%%%%%%%%%%%%%%%%%%%%%%%
	\section{Preliminaries}\label{sec:prelim}
	
	\subsection{General notations}\label{sec:gen_not}
	
	\paragraph{Inner-product spaces} Consider two inner-product spaces $(V,\g)$ and $(W,\h)$; we use the same notations $\g$ and $\h$ to denote the inner-products induced on the dual spaces $V^*$ and $W^*$.
	The inner-products $\g$ and $\h$ induce an inner-product $\g \otimes \h$ on $\Hom(V,W)\simeq V^* \otimes W$,
	and we denote its corresponding norm by $|\,\cdot\,|_{\g,\h}$. 
	For $A\in \Hom(V,W)$ and $\calC\subset \Hom(V,W)$,
	\[
	\dist_{\g,\h}(A,\calC) = \inf_{B\in\calC} |A-B|_{\g,\h}.
	\]
	We denote by $\O(\g,\h)$ the set of linear isometries
	from $V$ to $W$, and when $ \dim V=\dim W$, by
	$\SO(\g,\h)$ the set of orientation-preserving linear isometries from $V$ to $W$. 
	When choosing positively-oriented orthonormal frames in $V$ and $W$, $\O(\g,\h)$ and $\SO(\g,\h)$ reduce to the sets of matrices $\O(\dim V, \dim W)$ and $\SO(\dim V)$, respectively.
	
	\paragraph{Levi-Civita connection}
	Let $(\M,\g)$ be a Riemannian manifold. 
	We denote by $\nabla^\g$ its Levi-Civita connection, or by $\nabla$ when there is no ambiguity.
	We denote by $\frakX(\M)$ the set of all smooth vector fields over $\M$.
	
	\paragraph{Pullbacks} 
	Let $E\to\N$ be a vector bundle over $\N$; we denote by $ \Gamma(E) $ the set of sections of $ E $. 
	If $f:\M\to\N$ is a smooth map between two manifolds, then we denote by $f^*E\to\M$ the pullback bundle over $\M$, where for $p\in\M$, the fiber $(f^*E)_p$ is identified with the fiber $E_{f(p)}$.
	
	Given a covariant $ k-$tensor field $ A $ on $ \N $, the pullback of $ A $ by $ f $, denoted by $ f^*A $, is defined by
	\[ 
	(f^*A)_p(v_1,...v_k)=A_{f(p)}(df_p(v_1),..,df_p(v_k)),
	\]
	for every $ p \in \M $ and $ v_1,...,v_k \in T_p\M $. In particular, if $ \h $ is a Riemannian metric on $ \N $ and $ f:(\M,\g) \to (\N,\h) $ is a smooth immersion then $ f^*\h $ is a Riemannian metric on $ \M $, called the pullback metric induced by $ f $. Note that $ f $ is an isometric immersion if and only if $ f^*\h=\g $. In a similar way, given a manifold $ S $ and a map $ Q:T\N \to TS $, we denote by $ f^*Q:TM \to TS $ the map defined by
	\[ 
	(f^*Q)_p(v)=Q_{f(p)}(df_p(v)),
	\]
	for every $ p \in \M $ and $ v\in T_p\M $.
	
	\paragraph {Inequality constants}
	Throughout the proof we denote by $C$ a positive constant whose value may vary from line to line, but which only depends on fixed quantities, such as the geometry of manifolds or the value of the exponent $p$ in $L^p$ spaces.
	
	%%%%%%%%%%%%%%%%%%%%%%%%%%%%%%%%%%%%%%%%%%%%
	\subsection{Second fundamental form}
	\label{subsec:Second fundamental form}
	
	Let $(\tM,\tg)$ be a Riemannian manifold and $(\M,\g)$ an embedded Riemannian submanifold of $\tM$. Let $ \II: \frakX(\M)\times \frakX(\M) \to \Gamma(N\M) $ be given by:
	\[ 
	\II(X,Y)=P_\perp ({\nabla}^{\tilde{\g}}_X Y),
	\]
	where $N\M \subset T\tM|_\M$ is the normal bundle of $\M$ in $\tM$, $P_\perp: T\tM|_\M \to N\M$ is the normal projection, and ${\nabla}^{\tilde{\g}}$ is the Levi-Civita connection of $(\tM,\tg)$.
	Given a normal vector field $ \hn \in \Gamma(N\M) $ we define $ \b_\hn:\frakX(\M)\times \frakX(\M) \to C^{\infty}(\M) $ by
	\[ 
	\b_\hn(X,Y)=( \II(X,Y),\hn)_{\tg}.
	\]
	$ \b_\hn $ is called the \textbf{second fundamental form of $\M$ along $ \hn $} (some authors refer rather to $ \II $ as the second fundamental form). It can be shown that $\II$ is symmetric and therefore so is $ \b_\hn$. Let $ S_\hn: T\M \to T\M $ be the self-adjoint linear operator associated with $ \b_\hn $, defined by
	\[ 
	\b_\hn(X,Y)=(S_\hn(X),Y)_{\g}=(X,S_\hn(Y))_{\g}.
	\]
	$ S_\hn $ is called the \textbf{shape operator along $ \hn $} (or sometimes the \textbf{Weingarten map}). By the Weingarten equation \cite[Prop.~8.4]{Lee18},
	\[ 
	S_\hn(X)=-\Ppar({\nabla}^{\tilde{\g}}_X \hn),
	\]
	where $ \Ppar: T\tM|_\M \to T\M $ is the tangential projection.
	If the co-dimension of $\M$ in $\N$ is one, then ${\nabla}^{\tilde{\g}}_X \hn$ is automatically in $T\M$ (since $\hn$ has norm one), hence $\Ppar$ can be omitted.

	\paragraph{Second fundamental form induced by a smooth immersion}
	
	Let $ (\M,\g) $ be an oriented $ d- $dimensional Riemannian manifold and let $ (\N,\h)$ be an oriented $(d+1)$-dimensional Riemannian manifold.   
	Given a smooth immersion $ f:(\M,\g) \to (\N,\h) $, it is locally an embedding. Therefore, for each $ p\in \M $ there exists a neighborhood $ p\in U \subset \M $ such that $ f(U)\subset \N $ is an embedded submanifold. We define $ \hn_f: \M \to T\N $
	to be the unit vector field normal to $ f $ such that for every embedded submanifold $ f(U)\subset \N $, $ \hn_f $
	induces the same orientation on $ f(U) $ as the one induced by $ f $ (a unit vector field normal to $ f $ is unique up to a sign, and by choosing the orientation it induces it is determined uniquely).
	
	The immersion $ f $ defines a second fundamental form $ \b_f:\frakX(\M)\times \frakX(\M) \to C^{\infty}(\M) $ and a shape operator $ S_f:T\M \to T\M $ by
	\beq\label{eq:B_f_S_f} 
	\begin{split}
		\b_f(X,Y)&=\b_{\hn_f}(df(X),df(Y))\\
		&=(S_{\hn_f}(df(X)),df(Y))_{\h}\\
		&=(df^{-1} S_{\hn_f}(df(X)),Y)_{f^*\h}\\
		&=(S_f(X),Y)_{f^*\h}.
	\end{split}      
	\eeq
	By the Weingarten equation we conclude that $ S_f(X)=-df^{-1}({\nabla}^{\tilde{\h}}_{df(X)} \hn_f)$, where we identify $ \hn_f $ as a function from $ f(U) $ to $ T\N $.
	We will use this equation to define the shape operator for Sobolev immersions as well, see \eqref{eq:shape operator} below.
	
	For the special case where $\dim \M=d$ and $(\N,\h)=(\Rd,\euc)$ (i.e., the range of $f$ is the Euclidean space) we obtain that
	\[ 
	S_f=-df^{-1}\circ d\hn_f,
	\]
	where $ d\hn_f $ is the differential of the Gauss map.
	For further information regarding the second fundamental form and shape operator, see \cite{Lee18,Doc92}.

	%%%%%%%%%%%%%%%%%%%%%%%%%%%%%%%%%%%%%%%%%%%%
	
	\subsection{The double tangent $TT\N$ and the connector operator}
	\label{sec:TTN}
	
	For an immersion $f:\M\to \N$, the normal $\hn_f$ is a map $\M\to T\N$.
	In addition to its covariant derivative $\nabla \hn_f: T\M \to T\N$ discussed above, we also have its differential $d\hn_f : T\M \to TT\N$.
	This differential is the object naturally encountered when considering $\hn_f$ as a Sobolev function.
	In this subsection we summarize some known constructions regarding the double tangent space $TT\N$ that are needed in the analysis of Sobolev immersions.
	
	\paragraph{Covariant derivative defined by a connector}
	
	There is one-to-one correspondence between affine connections $\nabla$ on $T\N$ and vector bundle homomorphisms $K:TT\N\to T\N$, called the \emph{connector operator}; see \cite[Chapter~IV]{Michor2008}.
	For a manifold $ \M $, a smooth mapping $ \hn:\M \to T\N $, and a vector field $ X\in\frakX(\M) $ the \emph{covariant derivative of $ \hn $ along $ X $} can then be defined by
	\beq
	\label{eq:connector cd}
	\nabla_{X}\hn:=K\circ d\hn \circ X: \M \to T\M \to TT\N \to T\N.
	\eeq

	\paragraph{The Sasaki metric on $ T\N $}
	
	The connector (or affine connection) induces on $ TT\N $ a canonical decomposition $ TT\N=V(\N)\oplus H(\M) $, where $ V(\N)=\ker d\pi $ (where $\pi: T\N \to \N  $ is the projection), and $ H(M)=\ker K $. At every point $ (p,v)\in T\N $ the map
	\[ 
	(d\pi_{(p,v)}\times K_{(p,v)}): T_{(p,v)}T\N \to T_p\N \times T_p\N
	\]
	is an isomorphism. 
	
	\noindent The Sasaki metric \cite{sasaki1958}, $ \mathcal{S}_{\h,K} $ on $ T\N $ is the pullback of the metric $ \h $ on $ \N $ by the map $ (d\pi \times K) $:
	for every $ V,W\in TT\N $
	\beq
	\label{eq:Sasaki}
	(V,W)_{\mathcal{S}_{\h,K}}=(d\pi(V),d\pi(W))_{\h}+(K(V),K(W))_{\h}.
	\eeq
	See \cite[Chapter~3, Ex.~2]{Doc92} for a more explicit definition.

	%%%%%%%%%%%%%%%%%%%%%%%%%%%%%%%%%%%%%%%%%%%%
	\subsection{Sobolev spaces between manifolds}
	\label{subsec:sobolev}
	
	The following definitions and results are well-known; see \cite{Haj09} and \cite[Appendix~B]{Weh04} for proofs and further references.
	
	\noindent Let $\M, \mathcal{Q}$ be compact Riemannian manifolds, and let $D\in\bbN$ be large enough such that there exists an isometric embedding $\iota:\mathcal{Q}\to \R^D$ (Nash's theorem).
	For $p\in[1,\infty)$, we define the Sobolev space $W^{1,p}(\M;\mathcal{Q})$ by
	\[
	W^{1,p}(\M;\mathcal{Q}) := \BRK{u:\M\to \mathcal{Q}\,:\,\iota\circ u\in W^{1,p}(\M;\R^D)}. 
	\]
	This space inherits the strong and weak topologies of $W^{1,p}(\M;\R^D)$, which are independent of the embedding $\iota$.
	Although $du$ is only a weak derivative, it still holds that for almost every $ q\in\M $, $ du_q $ is a linear map from $ T_q\M $ to $ T_{u(q)}\mathcal{Q} $ \cite{CVS16}. Furthermore, $ du $ satisfies the chain rule, i.e., for every embedding $ \iota:Q \to \R^D $, $ d(\iota\circ u)=d\iota \circ du $ \cite[Prop.~1.9]{CVS16}. 
	
	\noindent It is worth mentioning that for every two Riemannian manifolds $ \M $ and $ \mathcal{Q} $ (not necessarily compact) $ W^{1,p}(\M;\mathcal{Q}) $ can be defined intrinsically, without the use of an embedding into Euclidean space. 
	When $ \M $ and $ \mathcal{Q} $ are compact these two definitions agree \cite{CVS16}.
	
	Note that $ W^{1,p}(\M;\mathcal{Q}) $ is generally not a vector space and therefore the notion of a norm does not exist. 
	However, the embedding $\iota$ induces a natural metric on $W^{1,p}(\M;\mathcal{Q})$ by
	\[ 
	d_{W^{1,p}(\M;\mathcal{Q})}(F,\tF):=\Norm{\iota \circ F - \iota \circ \tF}_{W^{1,p}(\M;\R^D)}.
	\]
	The metric itself depends on $ \iota $ but the topology it induces does not.
	
	%%%%%%%%%%%%%%%%%%%%%%%%%%%%%%%%%%%%%%%%%%%%
	\subsection{The space of Sobolev immersions}
	\label{subsec:sobolev_immersions}

	Let $ f\in W^{1,p}(\M;\N) $ be an immersion of a $d$-dimensional manifold $\M$ into a $(d+1)$-dimensional manifold $\N$. 
	Then for a.e.~$ p\in\M $, $ df_p:T_p\M \to T_{f(p)}\N $ is an injective linear map. Since $ \M $ is an oriented manifold, $ df_p $ induces an orientation on the codimension-$ 1 $ subspace $ df_p(T_p\M)\subset T_{f(p)}\N $. Since $ \N $ is an oriented manifold, there exists a unique unit vector $ \nu_f(p)\in T_{f(p)}\N  $ that is normal to $ df_p(T_p\M) $ and induces the same orientation on $ df_p(T_p\M)\subset T_{f(p)}\N $. We define $ \hn_f:\M \to T\N $, the normal vector field to $ f $, by $ \hn_f(p)=(f(p),\nu_{f}(p)) $. We usually identify $ \hn_f $ and $ \nu_f $, emphasizing the difference only when relevant.

	We endow $ TT\N $ with the Sasaki metric $ \mathcal{S}_{\h,K}$ induced by $ \h $ and the Levi-Civita connection. We define
	\[ 
	\Imm_p(\M;\N):=\BRK{f\in W^{1,p}(\M;\N) ~:~ \rank \, {df}= d \,\, \text{a.e.}, \; \hn_f \in W^{1,p}(\M;T\N)},
	\]
	with respect to the metrics $ \g,\h $ and $ \mathcal{S}_{\h,K} $ on $ \M,\N $ and $ TT\N $ respectively.
	Since $ \hn_f \in S\N \subset T\N $ (where $ S\N $ is the sphere bundle which is compact) then both definitions of $ \hn_f \in W^{1,p}(\M;T\N) $, the intrinsic and by embedding into Euclidean space, agree.
	
	We define the weak covariant derivative of $ \hn_f $ by \eqref{eq:connector cd}, i.e.,
	\[ 
	\nabla_{X}\hn:=K\circ d\hn_{f} \circ X: \M \to T\N,
	\]
	for every $ X\in\frakX(\M) $.
	By the definition of the Sasaki metric \eqref{eq:Sasaki}, $ \nabla_{X}\hn\in L^p(\M;T\N) $.
	Using the Weingarten equation we define the shape operator $ S_f:T\M \to T\M $ by
	\beq
	\label{eq:shape operator}
	S_f(X)=-df^{-1}\circ \nabla_{X}\hn_{f}.
	\eeq
	Thus, the energy \eqref{eq:E_pMN} is well-defined and finite for any $f\in \Imm_p(\M;\N)$.
	
	%%%%%%%%%%%%%%%%%%%%%%%%%%%%%%%%%%%%%%%%%%%%%%%%%
	\section{Proof of \thmref{THM:MAIN_CONST_CURV}}
	\label{sec:proof_main_const_curv}
	
	We divide the proof into six steps:
	\begin{description}
		\item{\emph{Step~I}:} We thicken $\M$ into $ \tM=\M \times [-\e,\e] $ for an appropriate $\e>0$ which will be fixed later on;
		we define an extension operator taking an immersion $ f:\M \to \N $, extending it along normal geodesics into a map $F:\tM \to \N$. We calculate the differential $dF$ of the extension.
		
		\item{\emph{Step~II}:}  We endow $\tM$ with a metric $G$, which in particular satisfies that $(\M,\g)$ is isometric to $ (\M\times\{0\},G|_{\M\times\{0\}})$, and $ \b $ is the second fundamental form of $\M$ in $\tM$.
		
		\item{\emph{Step~III}:} We prove that $ \dist_{G,\h}(dF_n,SO(G,\h)) \to 0  $ in $ L^p(\tM) $. Then, using {\cite[Thm.~3]{KMS19}} we obtain that there exists a local isometry $ F:(\tM,G) \to (\N,\h) $ such that $ F_n \to F $ in $ W^{1,p}(\tM;\N) $.
		
		\item{\emph{Step~IV}:} We show that $f:=F|_{\M \times \{0\}}$ is an isometric immersion of $\M$ in $\N$, such that $\b$ is its second fundamental form. Moreover, $F$ is an extension of $f$.
		
		\item{\emph{Step~V}:} We prove that $ f_n \to f $ in $ W^{1,p}(\M;\N) $.
		
		\item {\emph {Step~VI}:} We prove that $ \hn_{f_n} \to \hn_{f} $ in $ W^{1,p}(\M;T\N) $ . 
		
	\end{description}
	
	%%%%%
	In the rest of this section we elaborate on each of these steps.
	
	\bigskip
	\noindent\emph{Step~I: Extending an immersion $f:M\to \N$ to $\tM =\M \times [-\e,\e]$ and calculating the differential of the extension.}
	\bigskip
	
	We think of $\tM$ as a ``thin" sheet whose mid-surface is $\M$. We do not take the limit $\e\to0$ as our proof hinges on a rigidity theorem in codimension $0$.
	For later use we define the projection
	\beq\label{eq:pi}
	\pi: \tM \to \M \qquad \pi(p,t)=p. 
	\eeq
	
	Given an immersion $f:\M \to \N$, we extend it to $F:\tM \to \N$ by
	\beq
	\label{eq:F_def}
	F(p,t)=\exp_{f(p)}(t\hn_f(p)),
	\eeq
	where  $\hn_f $  is the unit normal vector field to $f$ as defined in \secref{subsec:sobolev_immersions}.
	Since $ \N $ is a compact Riemannian manifold, by the Hopf-Rinow theorem it is geodesically complete and thus $F$ is well-defined.
	The extension is constructed such that for every $p\in\M$, the curve 
	\[
	\gamma^f_p : [-\e,\e]\to\N
	\]
	defined by $\gamma^f_p(t) = F(p,t)$ is a geodesic emanating from $f(p)$ with initial velocity $\hn_f(p)$.
	
	Throughout the proof we denote by $ u_1,u_2:[-\e,\e] \to \R $ the functions
	\beq\label{eq:u_1_u_2}
	u_1(t) =
	\begin{cases}
		\cos(\sqrt{\kappa}t) & \kappa> 0\\
		1 & \kappa=0\\
		\cosh(\sqrt{-\kappa}t) & \kappa< 0\\
	\end{cases}
	\qquad\text{and}\qquad
	u_2(t) =
	\begin{cases}
		\sin(\sqrt{\kappa}t) & \kappa> 0\\
		t & \kappa=0\\
		\sinh(\sqrt{-\kappa}t) & \kappa< 0\\
	\end{cases}
	\eeq
	where $\kappa$ is the sectional curvature of $(\N,\h)$.
	
	%%%%%%%%%%
	\begin{proposition}
		\label{dF}
		Let $ (p,t)\in \tM $ and $ (v,w)\in T_{(p,t)}\tM \simeq T_p\M \times \R $.
		Then
		\begin{equation}
		\label{eq:dF}
		dF_{(p,t)}(v,w)=P_{\gamma^f_p,t}\left[u_1(t)\, df_{p}(v)-u_2(t)\, df_{p}\circ S_f(v)+w\, \hn_f(p)\right],
		\end{equation}
		where $ P_{\gamma^f_p,t} $ is the parallel transport in $\N$ between time $0$ and $t$ along the geodesic $ \gamma^f_p$.
	\end{proposition}
	%%%%%%%%%
	
	%%%%%%%
	\begin{proof}
		First, we assume that $ f $ is a smooth immersion. We start with 
		\[
		dF_{(p,t)}(0,1) = \dot{\gamma}^f_p(t)=P_{\gamma^f_p,t} (\dot{\gamma}^f_p(0))=P_{\gamma^f_p,t} (\hn_f),
		\]
		where we used that a geodesic is a curve whose velocity  is parallel along itself.
		
		Let $ v\in T_p\M $, and let $ p(s) $ be a curve in $\M$ such that $ p(0)=p $ and $ \dot{p}(0)=v $. We define a family of geodesics $ \Gamma(s,t)=F(p(s),t)= \gamma^f_{p(s)}(t)$. Since $ \Gamma(s,t) $ is a family of geodesics, the vector field $ dF_{(p,t)}(v,0)=\left.\pd{}{s}\right|_{s=0}\Gamma(s,t) $ (as a function of $t$) is a Jacobi field $ J(t) $.
		We note that
		\[ 
		J(0)=\left.\pd{}{s}\right|_{s=0}\Gamma(s,0)=dF_{(p,0)}(v,0)=df_{p}(v).
		\]
		From the symmetry lemma for variation fields \cite[Lemma 6.2]{Lee18} we obtain that
		\[ 
		\begin{split}
		\dot{J}(0)&=\left.D_t\right|_{t=0}\left.\pd{}{s}\right|_{s=0}\Gamma(s,t)\\
		&=\left.D_s\right|_{s=0}\left.\pd{}{t}\right|_{t=0}\Gamma(s,t)\\
		&=\left.D_s\right|_{s=0}\hn_f(p(s))\\
		&=\nabla^{\h}_{df_p(v)}\hn_f\\
		&=-df_p \circ S_f(v),\\ 
		\end{split}
		\]
		where $ S_f $ is the shape operator associated with $f$ as defined in \eqref{eq:B_f_S_f}.
		We conclude that for every $ (p,t)\in\tM  $ and  $ v\in T_p\M $,
		\[ 
		dF_{(p,t)}(v,0)=J(t),
		\] 
		where $ J(t) $ is the Jacobi field along $ \gamma^f_p $ with initial conditions $ J(0)=df_{p}(v) $ and $ \dot{J}(0)=-df_p \circ S_f(v) $.
		
		Note that $ J $ is a normal Jacobi field. Indeed, 
		both $J(0)$ and $\dot{J}(0)$ are in the image of $df_p$, hence  $J(t)\perp \dot{\gamma}^f_p(t)$ for all $t$ \cite[Prop.~10.7]{Lee18}.
		
		\noindent In a Riemannian manifold of constant sectional curvature $\kappa$, for normal Jacobi fields along a unit-speed geodesic, the Jacobi equation takes the form \cite[Prop.~10.12]{Lee18}:
		\[ 
		\frac{D^2}{\partial t^2} J(t)=-\kappa J(t).
		\]
		A straightforward calculation gives the solution of this equation with the above-mentioned initial conditions,
		\[ 
		dF_{(p,t)}(v,0) = J(t)=P_{\gamma^f_p,t}\Brk{u_1(t) \, df_{p}(v)-u_2(t) \, df_{p}\circ S_f(v)},
		\]
		which completes the proof when $ f $ is smooth.
		
		We now prove that the same holds for $ f\in\Imm_p(\M;\N)$. Recall that $ \hn_f:\M \to T\N  $ is defined by $ \hn_f(p)=(f(p),\nu_f(p)) $ (in the first part of the proof we identified $ \hn_f $ with $ \nu_f $).  
		The map $F:\tM \to \N$ is a composition of the exponential map $\exp : T\N \to \N$ ($ \N $ is complete so $ \exp  $ is defined on all of $ T\N $) with the map $\chi:\tM \to T\N$ defined by $\chi(p,t) := (f(p),t \nu_f(p))$;
		Since $\hn_f\in W^{1,p}(\M;T\N)$, it follows that $\chi\in W^{1,p}(\tM;T\N)$.
		Since $t\in [-\e,\e]$ and $\nu_f(p)$ is a unit vector, the image of $ \chi $ lies in a compact subset $K$ of $T\N$.
		Since $\exp : T\N \to \N$ is smooth, it is Lipschitz when restricted to $K$, hence, it follows that the composition map $F=\exp\circ \chi$ is in $W^{1,p}(\tM;\N)$ \cite[Prop.~2.6]{CVS16}.
		Moreover, the chain rule holds \cite[Prop.~1.8]{CVS16}:
		\[
		dF_{(p,t)} = d\exp_{(f(p),t\hn_f(p))} \circ d\chi_{(p,t)}.
		\]
		Therefore,  $dF$ is well-defined when $ f\in\Imm_p(\M;\N)$. In addition, $ d\chi_{(p,t)} $ is uniquely determined by $ \hn_f(p) $ and $ (d\hn_f)_p $, and by the chain rule so does $ dF_{(p,t)} $.

		Let $ \tf:\M \to \N $ be a smooth immersion such that $ \tf(p)=f(p) $, $ d\tf_p=df_p $ (which implies that $ \hn_{\tf}(p)=\hn_{f}(p) $) and $ (d\hn_{\tf})_p=(d\hn_f)_p $. Let $ \tF $ be its extension. 
		On the one hand, $ d\tF_{(p,t)}=dF_{(p,t)} $; on the other hand, since $ \tf $ is smooth $ d\tF_{(p,t)} $ is given by $ \eqref{eq:dF} $.  We thus conclude that for every $ f\in\Imm_p(\M;\N) $, $ dF $ is well-defined and \eqref{eq:dF} holds, which completes the proof.
%		\qed
		\end{proof}
	%%%%%%%
	
	%%%%%
	\bigskip
	\noindent\emph{Step~II: Endowing $\tM$ with a metric $G$, which in particular satisfies that $(\M,\g)$ is isometric to $ (\M\times\{0\},G|_{\M\times\{0\}})$, and $ \b $ is the second fundamental form of $\M$ in $\tM$. }
	\bigskip
	
	The requirements that $G|_{T\M}=\g $ and $ \b_\M=\b  $ do not determine $G$ uniquely.
	We define $G$ to be the unique metric such that, \emph{assuming} there exists an isometric immersion $ f:(\M,\g) \to (\N,\h) $ with second fundamental form $ \b $, the extension $F:\tM\to\N$ of $f$ would be an isometric immersion.
	That is, $ G=F^*\h $, and from
	\eqref{eq:dF} we obtain that
	\begin{equation}
	\label{eq:G}
	G=u_1^2(t) \, \pi^*\g -2u_1(t)u_2(t) \, \pi^*\b +u_2^2(t) \, \pi^*S^*\g +dt^2,
	\end{equation}
	where $t\in [-\e,\e]$ is the normal coordinate and $S^*\g(u,v) =(S(u),S(v))_{\g}$.
	In a coordinate system $(p,t) = (p_1,\ldots,p_d,t)$ of $\tM = \M\times [-\e,\e]$, 
	$ G $ can be represented by the following block matrix
	\[ 
	G_{(p,t)}= \brk{ 
		\begin{array}{c|c} 
		u_1^2(t)\g -2u_1(t)u_2(t)\b + u_2^2(t) S^T\g S & 0 \\ 
		\hline 
		0 & 1 
		\end{array}},
	\] 
	where $\g$, $\b$ and $S$ in the upper left block are the $d\times d$ matrix representations of these tensors with respect to the coordinates $(p_1,\ldots,p_d)$.
	Clearly $G$ is symmetric, from the definition of $u_1(t)$, $u_2(t)$ and the compactness of $ \M $, it follows that by choosing $ \e $ small enough it is also positive definite and therefore a metric. 
	Note that $G$ only depends on $\g$ and $ \b $ and not on $f$ itself.
	That is, while the construction of $G$ is ``backward engineered" from the sought result, the resulting metric only depends on the given data. The fact that $ G $ does not depend on $ f $ itself holds because of the assumption that $(\N,\h)$ is of constant curvature (see \secref{sec:Discussion} for more details).
	
	Clearly, the metric induced by $G$ on the submanifold $\M \times \{0\}$ is $\g$, and thus $\M$ is isometric to $ \M \times \{0\}  $. In addition, since $b_\M$ is determined by the derivative of $ G $ along the normal direction at $ t=0 $, i.e., $ \pd{}{t}G|_{\M \times \{0\}}$, we obtain that $ \b_\M=\b $ (for a detailed proof see Appendix~\ref{sec:appA}).
	
	%%%%%
	\bigskip
	\noindent\emph{Step~III: The extensions $F_n$ of $f_n$ satisfy $\dist_{G,\h}^p(dF_n,SO(G,\h)) \to 0 $ in $L^p(\tM)$. Therefore, there exists a local isometry $F:(\tM,G) \to (\N,\h)$ with $F_n \to F$ in $W^{1,p}(\tM;\N)$.  }
	\bigskip
	
	For an immersion $f:\M\to\N$, we denote by $O(df):T\M \to T\N$  the map such that for every $ p $, $ O(df)_p:T_p\M \to T_{f(p)}\N $ is the orthogonal map closest to $ df_p $ with respect to the norm induced by $ \g $ and $ \h $.
	It is well-known (for a detailed proof see \cite{Kah11}) that pointwise $O(df)$ is the orthogonal part in the polar decomposition of $df$, that is,
	\beq
	O(df)=df(df^{T}df)^{-1/2},
	\label{eq:O(df)}
	\eeq
	where the transpose is with respect to the inner-products $\g$ and $\h$. 
	In addition, since $ f $ is an immersion, then $ O(df) $ is unique.
	
	%%%%%%%%%%
	\begin{proposition}\label{A_f in SO}
		For $ f\in\Imm_p(\M;\N) $, define a section $A_f$ of $T^*\tM\otimes f^*T\N$ as follows:
		for $ (p,t)\in \tM $ and $ (v,w)\in T_{(p,t)}\tM \simeq T_p\M \times \R $,
		\beq
		\label{eq:A_f}
		A_f(v,w) := P_{\gamma^f_p,t}\left[u_1(t)\, O(df_p)(v)-u_2(t)\, O(df_p)\circ S(v)+w\, \hn_f(p)\right].
		\eeq
		Then, $A_f$ is a section of $\SO(G,\h)$, i.e., $A_f$ is an orientation-preserving isometry at almost every point.
	\end{proposition}
	%%%%%%%%%
	
	%%%%%%%
	\begin{proof}
		The section $ A_f $ has the same structure as $dF$, with $ df $ and $ S_f $ replaced by $ O(df) $ and $ S $ respectively. We begin by proving that $ A_f\in\Gamma(O(G,\h)) $, i.e., that $ A_f $ is an isometry at every point. 
		A straightforward calculation shows that  $ A_f $ is indeed an isometry, i.e., $ (A_f(v,w),A_f(\xi,\eta))_{\h}=((v,w),(\xi,\eta))_G $ (for a detailed proof see Appendix~\ref{sec:appB}).
		For a fixed $ p $, $ A_f $ is smooth as a function of $ t $, in other words, it is smooth along the fiber $ {p}\times [-\e,\e] $. Since  $ A_f $ is a smooth isometry along each fiber and the fiber is connected, in order to prove that $ A_f\in \SO(G,\h) $ almost everywhere, it suffices to find one point on each fiber $ {p}\times [-\e,\e] $ at which $ A_f $ is orientation-preserving.  
		Consider a point $ (p,0)\in \M \times \{0\} \subset \tM  $. For $ (v,w)\in T_{(p,0)}\tM \simeq T_p\M \times \R $ we obtain
		\[ 
		(A_f)_{(p,0)}(v,w)=O(df_p)(v)+w\, \hn_f(p),
		\]
		which is orientation preserving by the definition of $\hn_f$ (see \secref{subsec:sobolev_immersions}), since by \eqref{eq:O(df)}, $ O(df) $ induces the same orientation as $ df $ on the co-dimension 1 subspace perpendicular to $\hn_f$.
%		\qed
		\end{proof}
	%%%%%%%
	
	We next introduce a second metric $G'$ on $\tM$ having a product structure, 
	\[
	G' = \pi^*\g + dt^2,
	\]
	which is an approximation to $G$ for which we can apply Fubini's theorem.
	From the product structure of $G'$ we obtain that for every map $ Q:T\M \to T\N $,
	\beq
	\label{eq:G to g norm}
	{|\pi^* Q|}_{G',\h}={|Q|}_{\g,\h}\circ\pi.
	\eeq
	
	From the compactness of $\tM$ we obtain the following bounds:
	
	\begin{lemma}
		\label{norm and volume equiv}
		There exist constants $c_1,c_2>0$, such that for every linear map $A: T\tM\to T\N$,
		\[
		c_1 |A|_{G,\h} \le |A|_{G',\h} \le c_2 |A|_{G,\h},
		\]
		and there exist constants $c_3,c_4>0$, such that
		\[
		c_3 \le \frac{\Vol_G}{\Vol_{G'}} \le c_4.
		\]
	\end{lemma}
	
	%%%%%%%%%%
	\begin{proposition}
		\label{SO dist bound}
		Let $ f:\M \to \N $ be a smooth immersion and let $ F:\tM \to \N $ be its extension to $\tM$ as defined by \eqref{eq:F_def}. Then, 
		\[
		\int_{\tM} \dist_{G,\h}^p(dF,\SO(G,\h))\,\Vol_G \leq C\, \E_p(f).
		\]
	\end{proposition}
	%%%%%%%%%
	
	%%%%%%%
	\begin{proof}
		Since $ A_f\in \SO(G,\h))  $,
		\[ 
		\begin{split}
		\dist_{G,\h}^p(dF,\SO(G,\h))&\leq {|{dF-A_f}|}_{G,\h}^p\\ 
		&\hspace{-2cm}\leq C {|{dF-A_f}|}_{G',\h}^p \\
		&\hspace{-2cm}=C {\left|u_1(t)\pi^*(df-O(df)) -u_2(t)\pi^*(df\circ S_f -O(df)\circ S) \right|}_{G',\h}^p\\
		&\hspace{-2cm}\leq C \Brk{{\left|u_1(t)\pi^*(df-O(df)) \right|}_{G',\h}^p
			+{\left|u_2(t)\pi^*(df\circ S_f -O(df)\circ S)\right|}_{G',\h}^p}\\
		&\hspace{-2cm}\leq C \Brk{{\left|u_1(t)\pi^*(df-O(df)) \right|}_{G',\h}^p
			+{\left|u_2(t)\pi^*(df\circ S_f-df\circ S) \right|}_{G',\h}^p
			+{\left|u_2(t)\pi^*((df-O(df))\circ S) \right|}_{G',\h}^p}\\
		&\hspace{-2cm}= C \Brk{{\left|u_1(t)(df-O(df)) \right|}_{\g,\h}^p
			+{\left|u_2(t)(df\circ S_f-df\circ S) \right|}_{\g,\h}^p
			+{\left|u_2(t)((O(df)-df)\circ S) \right|}_{\g,\h}^p}\circ \pi \\
		&\hspace{-2cm}\leq C \Brk{{\left|u_1(t)(df-O(df)) \right|}_{\g,\h}^p
			+{\left|u_2(t) (df\circ S_f-df\circ S) \right|}_{\g,\h}^p
			+ \left|u_2(t)(O(df)-df)\right|^p_{\g,\h}  |S|_{\g,\g}^p}\circ \pi \\
		&\hspace{-2cm}\leq C \Brk{{\left|df-O(df)\right|}_{\g,\h}^p
			+{\left|df\circ (S_f- S)\right|}_{\g,\h}^p}\circ \pi.
		\end{split}
		\]
		In the passage to the second line we used the equivalence of the norms $G$ and $G'$;
		in the passage to the third line we substituted the definitions of $dF$ and $ A_f $, used the fact that parallel transport is a linear isometry and the fact that the normal components of $dF$ and $ A_f $ coincide;
		in the passage to the fourth and fifth lines we used the triangle inequality together with the convexity of $x\mapsto x^p $ ;
		in the passage to the sixth line we used \eqref{eq:G to g norm};
		the passage to the seventh line is immediate;
		in the passage to the last line we used the smoothness of  $ S $ and the compactness of $\M$ to obtain a global bound on $ {|S|}_{\g,\g}^p $, as well as the boundedness of $ u_1,u_2 $ on $ [-\e,\e] $.
		
		We conclude that
		\[ 
		\begin{split}
		&\int_{\tM} \dist_{G,\h}^p(dF,\SO(G,\h)) \, \Vol_G\\
		&\qquad\leq  C \int_{\tM}\Brk{{|df-O(df)|}_{\g,\h}^p
			+{|df\circ (S_f- S)|}_{\g,\h}^p}\circ \pi \, \Vol_G \\
		&\qquad\leq C \int_\M \int_{-\e}^{\e} \Brk{{|df-O(df)|}_{\g,\h}^p
			+{|df\circ (S_f- S)|}_{\g,\h}^p}\circ \pi \, dt \, \Vol_{\g} \\
		&\qquad\leq C \int_\M  \brk{{|df-O(df)|}_{\g,\h}^p
			+{|df\circ(S-S_f)|}_{\g,\h}^p} \, \Volume\\
		&\qquad= C\E_p(f),
		\end{split}
		\]
		where in the passage to the third line we used the equivalence of $\Vol_G$ and $\Vol_{G'}$ and then Fubini's theorem;
		in the passage to the fourth line we used the independence of the integrand on $t$. 
%		\qed
		\end{proof}
	%%%%%%%
	
	From the above result and the fact that $ \E_p(f_n)\to 0 $ we deduce that
	\[ 
	\dist_{(G,\h)}(dF_n,SO(G,\h)) \to 0 \quad \text{in } L^p(\tM).
	\]
	By {\cite[Thm.~3]{KMS19}}, there exists a subsequence (not relabeled) $F_n$ converging in $W^{1,p}((\tM,G);(\N,\h))$ to a limit $F$, which is a smooth local isometry $(\tM,G)\to(\N,\h)$.
	
	%%%%%
	\bigskip
	\noindent\emph{Step~IV: 
		$f:=F|_{\M \times \{0\}}$ is an isometric immersion of $\M$ in $\N$, such that $\b$ is its second fundamental form, and $F$ is its normal extension.}
	
	\bigskip
	
	%%%%%%%%%%
	\begin{proposition}
		\label{F ext f}
		Let $f:\M\to\N$ be the restriction of $F$,
		\[
		f = F|_{\M\times\{0\}}.
		\] 
		Then, $\g$ and $\b$ are the first and second fundamental forms of $f$. Moreover, $F$ is the normal extension of $f$, namely,
		\begin{equation}\label{eq:F ext f}
		F(p,t)=\exp_{f(p)}(t\hn_f(p)).
		\end{equation}
	\end{proposition}
	%%%%%%%%%
	
	%%%%%%%
	\begin{proof}
		The mapping $f$ is an isometric immersion as a restriction of the local isometry $F$ to the submanifold $\M\times\{0\}$. Since $\M$ is isometric to $ \M\times\{0\} $ we obtain that $\g$ is the first fundamental form of $f$. We proved that $ \b_\M=\b $, which combined with the fact that $F$ is a local isometry implies that $ \b_f=\b $.
		
		Let $\tF:(\tM,G) \to (\N,\h) $  be the normal extension of $f$,  $ \tF(p,t)=\exp_{f(p)}(t\hn_f(p)) $; we will show that $ \tF=F $. As stated above, $ \b_f=\b $, which in turn implies that the corresponding shape operator, $ S_f $, is equal to $ S $. In addition, since $f$ is an isometric immersion,  $ O(df)=df $. Thus comparing \eqref{eq:dF} and \eqref{eq:A_f}, we obtain that:
		\[
		d\tF=A_f\in SO(G,\h),
		\]
		which means that $ \tF $ is an orientation-preserving local isometry.  
		
		Since $ F,\tF:\tM \to \N $ are local isometries and $\tM$ is connected, it suffices to show that 
		\[
		\tF(p,0)=F(p,0) \Textand d\tF_{(p,0)} =dF_{(p,0)}. 
		\]
		at a single point $ (p,0) \in \M \times \{0\} \subset \tM  $ 
		to conclude that $F=\tF$ everywhere \cite[Prop.~5.22]{Lee18}.
		
		The first condition is immediate from the definitions of $f$ and $\tF$, since
		\[
		\tF(p,0)=\exp_{f(p)}(0)=f(p)=F(p,0).
		\]
		As for the second condition, since $d\tF_{(p,0)}$ and $dF_{(p,0)}$ are orientation-preserving isometries, it suffices to show that they coincide on $\operatorname{Im}( d\pi) \simeq T_p\M$, since it  is a subspace of codimension $1$. Indeed, from \eqref{eq:dF},
		\[
		d\tF_{(p,0)}(v,0) =P_{\gamma^f_p,0}\Brk{u_1(0)\, df_p(v)-u_2(0)\, df_p(S_f(v))+0\, \hn_f(p)} = df_p(v)=dF_{(p,0)}(v,0),
		\]
		which completes the proof.
%		\qed
		\end{proof}
	%%%%%%%
	
	%%%%%
	\bigskip
	\noindent\emph{Step~V: $ f_n \to f $ in $ W^{1,p}(\M;\N)$.}
	\bigskip
	
	As described in \secref{subsec:sobolev}, a metric on $W^{1,p}(\M;\N)$ is defined via an isometric embedding $\iota:(\N,\h) \to (\R^D,\euc)$. 
	To this end, we fix such an embedding throughout this section.
	
	%%%%%%%%%%
	\begin{proposition}
		\label{averaged_seq}
		There exists a sequence $ \e_n \to 0 $, $\e_n\in(0,\e]$, such that
		\[ 
		\frac{1}{2\e_n}\int_\M\int_{-\e_n}^{\e_n} \brk{|\iota\circ F -\iota \circ F_n|_\euc^p + |d(\iota\circ F - \iota \circ F_n)|_{G',\euc}^p} dt \,\Vol_{\g} \to 0.
		\] 
	\end{proposition}
	%%%%%%%%%
	
	%%%%%%%
	\begin{proof}
		By definition,
		\[
		\brk{d_{W^{1,p}(\M;\N)}(F_n,F)}^p = \int_{\tM} \brk{|\iota\circ F -\iota \circ F_n|_\euc^p + |d(\iota\circ F - \iota \circ F_n)|_{G,\euc}^p}\, \Vol_G.
		\]
		Since $F_n\to F$ in $W^{1,p}(\M;\N)$, we may define, for $n$ large enough,
		\[ 
		\e_n=
		\begin{cases}
		\brk{d_{W^{1,p}(\M;\N)}(F_n,F) }^{p/2}  &  d_{W^{1,p}(\M;\N)}(F_n,F)\neq 0\\
		\frac{\e}{n} &  d_{W^{1,p}(\M;\N)}(F_n,F)=0,
		\end{cases}
		\]
		which satisfies the required properties.
		Using the equivalence of the metrics $G$ and $G'$ and that $ \e_n\leq \e $,
		\[
		\begin{split}
		&\frac{1}{2\e_n}\int_\M\int_{-\e_n}^{\e_n} \brk{|\iota\circ F -\iota \circ F_n|_\euc^p + |d(\iota\circ F -\iota \circ F_n)|_{G',\euc}^p} dt \,\Vol_{\g}\\
		&\qquad \le \frac{1}{2\e_n}\int_{\tM} \brk{|\iota\circ F -\iota \circ F_n|_\euc^p + |d(\iota\circ F -\iota \circ F_n)|_{G',\euc}^p} \Vol_{G'}\\
		&\qquad \leq \frac{C}{2\e_n}\int_{\tM} \brk{|\iota\circ F -\iota \circ F_n|_\euc^p + |d(\iota\circ F - \iota \circ F_n)|_{G,\euc}^p} \Vol_G\\
		&\qquad = \frac{C}{2\e_n} \brk{d_{W^{1,p}(\M;\N)}(F_n,F) }^p\\
		&\qquad \leq C\brk{d_{W^{1,p}(\M;\N)}(F_n,F) }^{p/2}\to 0.
		\end{split} 
		\]
%		\qed
		\end{proof}
	%%%%%%%
	
	We now prove that $f_n \to f$ in  $ W^{1,p}(\M;\N) $; first, we prove convergence  in  $ L^p(\M;\N) $.
	Denote $R_n=\iota \circ F_n - \iota \circ f_n\circ \pi$, where $\pi:\tM\to \M$ is the projection defined in \eqref{eq:pi}. 
	Denote by $ d_{\h} $ the distance function induced by $\h$ on $ \N $. Since $ \iota:(\N,\h)\to (\R^D,\euc)  $ is an isometric embedding it preserves the length of curves, which implies that $ |\iota(x)-\iota(y)|_\euc\leq d_{\h}(x,y) $ for every $ x,y\in \N $. From the definitions of $ \exp_{f_n(p)} $ and $ d_{\h} $, we obtain
	\[ 
	|R_n(p,t)|_\euc=\Abs{\iota \circ F_n(p,t) - \iota \circ f_n(p)}_\euc\leq d_{\h}\brk{\exp_{f_n(p)}(t\hn_{f_n}(p)),f_n(p)}=|t|.
	\]
	The same estimate holds for $R = \iota \circ F - \iota \circ f \circ\pi$.
	Hence,
	\[ 
	\begin{split}
	&\int_\M \Abs{\iota \circ f - \iota \circ f_n}_\euc^p \,\Vol_{\g}\\
	&\qquad = \frac{1}{2\e_n} \int_\M \int_{-\e_n}^{\e_n} \Abs{\iota \circ f\circ\pi - \iota \circ f_n \circ\pi}_\euc^p\, dt \,\Vol_{\g}\\
	&\qquad \leq \frac{C}{2\e_n} \int_\M \int_{-\e_n}^{\e_n} \brk{\Abs{\iota \circ F - \iota \circ F_n}_\euc^p
		+ \Abs{R_n}_\euc^p + \Abs{R }_\euc^p} \, dt \,\Vol_{\g}\\
	&\qquad \leq \frac{C}{2\e_n} \int_\M \int_{-\e_n}^{\e_n} \brk{\Abs{\iota \circ F - \iota \circ F_n}_\euc^p
		+ |t|^p} \, dt \,\Vol_{\g}.
	\end{split}
	\]
	From \propref{averaged_seq} we conclude that $ \int_\M \Abs{\iota \circ f - \iota \circ f_n}_\euc^p \,\Vol_{\g} \to 0  $ as desired.
	\smallskip
	
	Next, we prove the convergence $df_n \to df$ in $ L^p $ (i.e., the convergence $ d(\iota \circ f_n)\to d(\iota \circ f) $ in $ L^p(\M;T^*\M\otimes \R^D) $).
	We extend the projection $P_{\parallel}: T\tM|_\M \to T\M$ to a map $T\tM \to T\tM $ by $P_{\parallel}(v,w)=(v,0)$.
	Using Fubini's theorem and \eqref{eq:G to g norm}, we estimate the distance between $df_n$ and $df$ as follows:  
	\[ 
	\begin{split}
	&\int_\M \Abs{d(\iota\circ f) - d(\iota \circ f_n)}_{\g,\euc}^p \,\Vol_{\g}\\
	&\qquad = \frac{1}{2\e_n} \int_\M \int_{-\e_n}^{\e_n} \Abs{\pi^* (d\iota \circ df - d\iota \circ df_n)}_{G',\euc}^p\, dt \,\Vol_{\g}\\
	&\qquad = \frac{1}{2\e_n} \int_\M \int_{-\e_n}^{\e_n} \big|\pi^* (d\iota \circ df - d\iota \circ df_n) - (d\iota \circ dF-d\iota \circ dF_n)\circ P_{\parallel}\\
	&\qquad \qquad \qquad \qquad \qquad + (d\iota \circ dF-d\iota \circ dF_n)\circ P_{\parallel} \big|_{G',\euc}^p\, dt \,\Vol_{\g}\\
	&\qquad \leq C (I_1+I_2+I_3),\\
	\end{split}
	\]
	where
	\[ 
	\begin{split}
	I_1&=\frac{1}{2\e_n}\int_\M\int_{-\e_n}^{\e_n} | (d\iota \circ dF-d\iota \circ dF_n)\circ P_{\parallel}|_{G',\euc}^p \, dt \,\Vol_{\g}\\
	I_2&=\frac{1}{2\e_n}\int_\M\int_{-\e_n}^{\e_n} |\pi^* (d\iota \circ  df)-(d\iota \circ dF)\circ P_{\parallel}|_{G',\euc}^p \, dt \,\Vol_{\g}\\
	I_3&=\frac{1}{2\e_n}\int_\M\int_{-\e_n}^{\e_n} |\pi^* (d\iota \circ  df_n) - (d\iota \circ dF_n)\circ P_{\parallel}|_{G',\euc}^p \, dt \,\Vol_{\g}. \\
	\end{split}
	\]
	
	We evaluate each of these three integrals separately. From \propref{averaged_seq},
	\[ 
	\begin{split}
	\limsup_{n\to\infty} I_1
	&\leq \limsup_{n\to\infty}  \frac{1}{2\e_n}\int_\M\int_{-\e_n}^{\e_n}
	|d(\iota \circ F-\iota \circ F_n)|_{G',\euc}^p \, dt \,\Vol_{\g} = 0.
	\end{split}
	\]
	In order to evaluate $I_2$ and $I_3$ we use the following lemma:
	\begin{lemma}
		\label{par_aprox}
		Let $ \gamma:[-\e,\e] \to (\N,\h)  $ be a unit speed geodesic and let $ x\in T_{\gamma(0)}\N $. Define $ R_{\gamma,x}(t):[-\e,\e] \to \R^D $ by 
		\[
		R_{\gamma,x}(t)=(d\iota)_{\gamma(t)} \circ P_{\gamma,t}(x)-(d\iota)_{\gamma(0)} (x)
		\]
		then there exists a constant $ C $ (which does not depend on $\gamma$ or $x$) such that $ |R_{\gamma,x}(t)|_\euc \leq C|t||x|_{\h}$.
	\end{lemma}
	
	\lemref{par_aprox} can be proved by working in coordinates and implementing a Gronwall's inequality type argument on the set of parallel transport equations. For a detailed proof see Appendix~\ref{sec:appC}.
	
	Since $ \E_p(f_n) \to 0 $ and $\M$ is compact, then there exists a bound $ C $ such that for every $ n $
	\beq
	\label{eq:bounds on df,dfs_f}
	\int_\M |df_n|_{\g,\h}^p \Vol_{\g}\leq C \Textand \int_\M |df_n\circ S_{f_n}|_{\g,\h}^p \Vol_{\g}\leq C.
	\eeq
	The first inequality follows from the fact that $ O(\g,\h) $ is bounded and that the stretching term in $ \E_p(f_n) $ tends to zero. The second inequality follows from using the first inequality together with the fact that the bending term in $ \E_p(f_n) $ tends to zero.
	In addition, since $ \iota:(\N,\h) \to (\R^{D},\euc) $ is an isometric embedding we can choose $ C $ such that  $ |d\iota|_{\h,\euc}^p\leq C $.
	
	Consider now $ I_3 $. In order to compare $df_n$ and $dF_n$, we first compare $dF_n$ and the parallel transport of $df_n$ so their image would lie in the same tangent space---$ T_{F_n(p,t)}\N $.
	We then use Lemma~\ref{par_aprox} to estimate the distance between $df_n$ and its parallel transport.
	
	We estimate the distance between $dF_n$ and the parallel transport of $df_n$ by
	\[
	\begin{split}
	&\Abs{(d\iota \circ dF_n\circ P_{\parallel})_{(p,t)}-(d\iota)_{\gamma^{f_n}_{p}(t)} \circ P_{\gamma^{f_n}_{p},t} \circ (df_n)_p \circ (d\pi)_{p,t}}_{G',\euc}^p \\
	&=\Abs{(d\iota)_{F_n(p,t)}\circ\Brk{(dF_n\circ P_{\parallel})_{(p,t)}-P_{\gamma^{f_n}_{p},t} \circ (df_n)_p \circ (d\pi)_{p,t}}}_{G',\euc}^p\\
	&\qquad\qquad\qquad\qquad \leq  C|t|^p \brk{|(df_n)_p|_{\g,\h}^p + |(df_n)_p\circ (S_{f_n})_p|_{\g,\h}^p}\circ \pi.
	\end{split}
	\]
	This estimate follows from \eqref{eq:dF}, the fact that parallel transport is a linear isometry, the definition of $ u_1(t) $ and $ u_2(t) $, \eqref{eq:G to g norm} and the global bound on $ |d\iota|_{\h,\euc}^p $.
	It follows that 
	\[
	\begin{split}
	I_3 &\leq \frac{C}{2\e_n}\int_\M\int_{-\e_n}^{\e_n} \big(\Abs{\brk{(d\iota)_{f_n(p)} - (d\iota)_{\gamma^{f_n}_{p}(t)} \circ  P_{\gamma_{f_n(p)},t} } \circ  (df_n)_p\circ (d\pi)_{(p,t)}}_{G',\euc}^p \\
	&\qquad \qquad \qquad \qquad\qquad\quad\, +|t|^p \brk{|(df_n)_p|_{\g,\h}^p + |(df_n)_p\circ (S_{f_n})_p|_{\g,\h}^p}\circ \pi\big)\, \,dt \,\Vol_{\g}\\
	&\leq \frac{C}{2\e_n}\int_{-\e_n}^{\e_n}|t|^p \int_\M  \brk{|(df_n)_p|_{\g,\h}^p + |(df_n)_p\circ (S_{f_n})_p|_{\g,\h}^p}\circ \pi \,  \Vol_{\g}\,dt\\
	&\leq \frac{C}{2\e_n}\int_{-\e_n}^{\e_n}|t|^p \,dt \to 0,
	\end{split}
	\]
	where in the passage to the second line we used Fubini's theorem together with \lemref{par_aprox} and \eqref{eq:G to g norm}; and in the passage to the third line we used \eqref{eq:bounds on df,dfs_f}.
	
	$I_2$ can be evaluated in the exact same manner. To conclude
	\[ 
	\int_\M |d\iota \circ df - d\iota \circ df_n|_{\g,\euc}^p \,\Vol_{\g}\leq C(I_1 + I_2 +I_3 ) \to 0,
	\]
	which completes the proof that $f_n \to f$ in $W^{1,p}(\M;\N)$.

	%%%%%
	\bigskip
	\noindent\emph{Step~VI: $ \hn_{f_n} \to \hn_{f} $ in $ W^{1,p}(\M;T\N) $.}
	\bigskip
	
	We now complete the proof of \thmref{THM:MAIN_CONST_CURV}, by showing that $\hn_{f_n} \to \hn_{f} $ in $ W^{1,p}(\M;T\N) $, i.e. that for every isometric embedding $ i: (T\N,S_{\h} ) \to \R^K $ we obtain the convergence $ i\circ \hn_{f_n} \to i\circ \hn_{f}  $ in $ W^{1,p}(T\N;R^K) $.
	We do that by first showing that for every isometric embedding $ \iota:(\N,\h) \to \R^D $ we obtain the convergence $d\iota \circ \hn_{f_n} \to d\iota \circ \hn_{f} $ in $L^p$, and the convergence of the \emph{covariant} derivatives   $d\iota \circ\nabla\hn_{f_n} \to d\iota \circ\nabla\hn_{f} $ in $L^p$ (Proposition~\ref{nf convergence}).
	It is important to note that $ d\iota:TN \to \R^{2D} $ is not necessarily an isometric embedding of $ T\N $ (with respect to the Sasaki metric) and that the covariant derivative $ \nabla\hn_{f_n}: T\M \to T\N $ is not the full derivative $ d\hn_{f_n}:T\M \to TT\N $.
	We then prove that Proposition~\ref{nf convergence} together with the convergence $f_n \to f$ in $W^{1,p}(\M;\N)$, implies the convergence of $\hn_{f_n} \to \hn_{f} $ in $ W^{1,p}(\M;T\N) $. 
	This last part, being a general claim about convergence of vector fields in Sobolev spaces between manifolds, is postponed to Proposition~\ref{pn:appD} in Appendix~\ref{sec:appD}.

	%%%%%%%%%%
	\begin{proposition}
		\label{nf convergence}
		For every smooth isometric embedding $ \iota:(\N,\h) \to \R^D $,
		\[ 
		\begin{split}
		&d\iota \circ \hn_{f_n} \to d\iota \circ \hn_{f} \qquad \text{in } L^p(\M;\R^D)\\
		&d\iota \circ \nabla \hn_{f_n} \to d\iota \circ \nabla \hn_{f} \qquad  \text{in } L^p(\M;T^*\M \otimes \R^D).\\
		\end{split}
		\]
	\end{proposition}
	%%%%%%%%%
	
	%%%%%%%
	\begin{proof}
		Let $ \iota:\N \to \R^{D} $ be an isometric embedding, then
		\[ 
		\begin{split}
		&\int_{\M} |d\iota \circ \hn_{f_n} - d\iota \circ \hn_{f}|_{\euc}^{p} \,\Vol_{\g}\\
		&\qquad \qquad=\frac{1}{2\e_n}\int_\M\int_{-\e_n}^{\e_n} |(d\iota \circ \hn_{f_n} - d\iota \circ \hn_{f})\circ \pi|_{\euc}^{p} \, dt \,\Vol_{\g}\\
		&\qquad \qquad\leq C \frac{1}{2\e_n}\int_\M\int_{-\e_n}^{\e_n}\left( |(d\iota \circ \hn_{f_n})\circ \pi - (d\iota \circ P_{\gamma^{f_n}_{p},t}\circ\hn_{f_n})\circ \pi|_{\euc}^{p}\right.\\
		&\qquad \qquad \qquad \qquad \qquad \qquad \left.+|(d\iota \circ P_{\gamma^{f_n}_{p},t}\circ\hn_{f_n})\circ \pi-(d\iota \circ P_{\gamma^{f}_{p},t}\circ\hn_{f})\circ \pi |_{\euc}^{p}\right.\\
		&\qquad \qquad \qquad \qquad \qquad \qquad \left.+ |(d\iota \circ P_{\gamma^{f}_{p},t}\circ\hn_{f})\circ \pi - (d\iota \circ \hn_{f})\circ \pi |_{\euc}^{p} \right)	 \, dt \,\Vol_{\g} \\
		&\qquad \qquad\leq C\frac{1}{2\e_n}\int_\M\int_{-\e_n}^{\e_n} \brk{|t|^p|\hn_{f_n}|_{\h}^{p}
			+ |(d\iota \circ dF_n - d\iota \circ dF)(\partial_t)|_{\euc}^{p}
			+ |t|^p|\hn_{f}|_{\h}^{p}} \, dt \,\Vol_{\g} \\
		&\qquad \qquad\leq C\frac{1}{2\e_n}\int_\M\int_{-\e_n}^{\e_n}\brk{ |t|^p
			+ |(d\iota \circ dF_n - d\iota \circ dF)|_{G',\euc}^{p}|(\partial_t)|_{G'}^{p}
			+ |t|^p} \, dt \,\Vol_{\g} \\
		&\qquad \qquad\leq C\frac{1}{2\e_n}\int_\M\int_{-\e_n}^{\e_n} \brk{2|t|^p
			+ |(d\iota \circ dF_n - d\iota \circ dF)|_{G',\euc}^{p}}\, dt \,\Vol_{\g}  \to 0 \, ,\\
		\end{split}
		\]
		where in the second inequality (fifth line) we use \lemref{par_aprox}, \eqref{eq:dF} and \eqref{eq:F ext f};
		in the third inequality we use that fact that $ \hn_{f_n} $ and $ \hn_{f} $ are unit vector fields;  in the fourth inequality that $ |\partial_t|_{G'}=1 $; Finally, in the limit we use \propref{averaged_seq} and the fact that $ \M $ is compact.
		
		From \eqref{eq:shape operator} we obtain that $ df_n\circ S_{f_n}=-\nabla \hn_{f_n} $ and $ df\circ S_f=-\nabla \hn_f $. Since $ \b $ is the second fundamental form of $ f $ we know that $ S_f=S $. Thus,
		\[ 
		\begin{split}
		&\int_{\M} |d\iota \circ \nabla \hn_f - d\iota \circ \nabla \hn_{f_n} |_{\g,\euc}^{p} \,\Vol_{\g}\\
		&\qquad \qquad =\int_{\M} |d\iota \circ df_n\circ S_{f_n} - d\iota \circ df\circ S|_{\g,\euc}^{p} \,\Vol_{\g}\\
		& \qquad \qquad \leq C \int_{\M} \brk{ |d\iota \circ df_n\circ S_{f_n} - d\iota \circ df_n\circ S|_{\g,\euc}^{p}
			+ |d\iota \circ df_n\circ S - d\iota \circ df\circ S|_{\g,\euc}^{p}}  \,\Vol_{\g}\\
		& \qquad \qquad = C \int_{\M} \brk{ |d\iota \circ( df_n\circ S_{f_n} -df_n\circ S)|_{\g,\euc}^{p}
			+ |(d\iota \circ df_n - d\iota \circ df)\circ S|_{\g,\euc}^{p}}  \,\Vol_{\g}\\
		& \qquad \qquad \leq C \int_{\M} \brk{ |d\iota|_{\h,\euc}^{p} |df_n\circ S_{f_n} -df_n\circ S|_{\g,\h}^{p}
			+ |d\iota \circ df_n - d\iota \circ df|_{\g,\euc}^{p} |S|_{\g,\g}^{p}}  \,\Vol_{\g}\\
		& \qquad \qquad \leq C \brk{\E_p(f_n) 
			+ \int_{\M}  |d\iota \circ df_n - d\iota \circ df|_{\g,\euc}^{p} \,\Vol_{\g}} \to 0 \, ,\\
		\end{split}
		\]
		where in the passage to the last line we used the compactness of $ \M $ and $ \N $ to obtain global bounds on $ |S|_{\g,\g} $ and $ |d\iota|_{\h,\euc} $ respectively; Finally, the limit is obtained since $ \E_p(f_n) \to 0 $ and $ f_n \to f $ in $ W^{1,p}(\M;\R^{D}) $.
%		\qed
		\end{proof}
	%%%%%%%

	As proved in Appendix~\ref{sec:appD}, \propref{nf convergence} combined with $ f_n \to f $ in $ W^{1,p}(\M;\N) $ implies that $ \hn_{f_n} \to \hn_f  $ in $ W^{1,p}(\M;T\N) $, which completes the proof of \thmref{THM:MAIN_CONST_CURV}.

	%%%%%%%%%%%%%%%%%%%%%%%%%%%%%%%%%%%%%%%%%%%%%%%%%
	\section{Proof of \thmref{thm:main_euc}}
	\label{sec:proof_main_euc}
	
	The proof of \thmref{thm:main_euc} follows the exact same lines as the proof of \thmref{THM:MAIN_CONST_CURV}, and is, in fact, simpler in some places, hence we will only emphasize the differences. We divide the proof into the same six steps as in the proof of \thmref{THM:MAIN_CONST_CURV}.
	
	%%%%%
	\bigskip
	\noindent\emph{Step~I: Extending immersions $f:M\to \Rd$ to $\tM =\M \times [-\e,\e]$ and calculating the differential of the extension.}
	\bigskip
	
	In Euclidean space, $\exp_{f_n(p)}(t\hn_f(p))=f(p)+t\hn_f(p) $. 
	I.e., the geodesic is a straight line through $f(p)$ with velocity $\hn_f(p)$. Therefore, given an immersion $ f:\M \to \Rd $ we extend it to $ F:\tM \to \Rd $ by
	\beq
	\label{eq:F_def_euc}
	F=f\circ\pi +t\,\hn_f \circ \pi.
	\eeq
	Differentiating,
	\beq
	\begin{split}
		dF &= \pi^* df + t\, \pi^* d\hn_f + dt \otimes\brk{\hn_f\circ \pi}\\
		&= \pi^* df -t\, \pi^*\brk{df\circ S_f} + dt \otimes\brk{\hn_f\circ \pi},\\
	\end{split}
	\label{eq:dF_euc}
	\eeq
	where, as mentioned in \secref{subsec:Second fundamental form}, we used that $ d\hn_f = -df\circ S_f $. Taking into account that in Euclidean space parallel transport is trivial we obtain that \eqref{eq:dF_euc} is just a special case of \eqref{eq:dF}.
	
	%%%%%
	\bigskip
	\noindent\emph{Step~II: Endowing $\tM$ with a metric $G$, which in particular satisfies that $(\M,\g)$ is isometric to $ (\M\times\{0\},G|_{\M\times\{0\}})$ and $ \b $ is the second fundamental form of $\M$ in $\tM$.}
	\bigskip
	
	In this case, the metric $G$ is given by
	\[
	G=\pi^*\g -2t \, \pi^*\b +t^2\pi^*S^*\g +dt^2,
	\]
	which is simply \eqref{eq:G} with $ u_1(t)=1,u_2(t)=t $.
	The proof that $G$ is a metric for small enough $\e>0$, and that the map $\M \to \M\times \{0\}$ is an isometric embedding with second fundamental form $\b$ is the same as in the proof of \thmref{THM:MAIN_CONST_CURV}.
	
	%%%%%
	\bigskip
	\noindent\emph{Step~III: The extensions $F_n$ of $f_n$ satisfy $\dist_{G,\euc}^p(dF_n,SO(G,\euc)) \to 0 $ in $L^p(\tM)$. Therefore, there exists a local isometry $F:(\tM,G) \to (\Rd,\euc)$ with $F_n \to F$ in $W^{1,p}(\tM;\Rd)$ (modulo translations).}
	\bigskip
	
	As in step~II, the proof is almost identical to the proof of step~III in \thmref{THM:MAIN_CONST_CURV}. The only small difference in the proof is that after proving that
	\[ 
	dist_{(G,\euc)}(dF_n,SO(G,\euc)) \to 0 \quad \text{in } L^p(\tM),
	\]
	instead of  using \cite[Thm.~3]{KMS19},
	we use \cite[Cor.~7]{KMS19}, which covers the case where the target is Euclidean space. It immediately follows that there exists a subsequence of $F_n$ (not relabeled), and constants $c_n\in \R^{d+1}$, such that $F_n-c_n$ converges in $W^{1,p}((\tM,G);(\Rd,\euc))$ to a limit $F$, which is a smooth local isometry $(\tM,G)\to(\Rd,\euc)$. 
	The translations $c_n$ are needed because of the non-compactness of $\R^{d+1}$, and can be taken to be the averages of $F_n$.

	%%%%%
	\bigskip
	\noindent\emph{Step~IV: $f:=F|_{\M \times \{0\}}$ is an isometric immersion of $\M$ in $\Rd$, such that $\b $ is its second fundamental form, and $F$ is its normal extension.}
	\bigskip
	
	The proof of this step is analogous to the proof of Step~IV in \thmref{THM:MAIN_CONST_CURV} (Proposition~\ref{F ext f}).
	Note that in this case, $F$, as an extension of $f$ (see \eqref{eq:F ext f}), has a particularly simple form,
	\beq
	\label{eq:F_ext_f_euc}
	F=f\circ\pi +t\hn_f \circ \pi.
	\eeq
	Finally, since $\g$ and $ \b $ are the first and second fundamental forms of $f$ then they must satisfy the (GCM) compatibility equations.
	
	%%%%%
	\bigskip
	\noindent\emph{Step~V: $f_n \to f$ in $W^{1,p}(\M;\Rd)$, modulo translations.}
	\bigskip
	
	The proof is again similar to the proof in \thmref{THM:MAIN_CONST_CURV}, however simpler. Note that $ W^{1,p}(\M;\Rd)$ is a normed space (unlike $ W^{1,p}(\M;\N)$), hence there is no need to introduce an isometric embedding $ \iota $ (in this case $ \iota=Id:\Rd \to \Rd $).
	In particular, Proposition~\ref{averaged_seq} takes the simpler form:
	%%%%%%%%%%
	\begin{proposition}
		\label{averaged_seq_euc}
		There exists a sequence $ \e_n \to 0 $, $\e_n\in(0,\e]$, such that
		\[ 
		\frac{1}{2\e_n}\int_\M\int_{-\e_n}^{\e_n} \brk{|F - (F_n-c_n)|_\euc^p + | dF - dF_n|_{G',\euc}^p} dt \,\Vol_{\g} \to 0,
		\] 
		where $c_n$ are the averages of $F_n$ over $\tM$.
	\end{proposition}
	%%%%%%%%%
	
	We now prove that $ f_n -c_n \to f $ in $ W^{1,p}(\M;\Rd)$.
	Because of the simplicity of the proof
	when the ambient space is Euclidean, we expand it below for the benefit of the readers that are mostly interested in Euclidean ambient spaces.
	Without loss of generality, we will henceforth assume that $c_n=0$.
	We start with the convergence of $ f_n \to f $ in  $ L^p(\M;\R^d) $.
	\[ 
	\begin{split}
	\int_\M \Abs{f - f_n}_\euc^p \,\Vol_{\g} &= \frac{1}{2\e_n} \int_\M \int_{-\e_n}^{\e_n} \Abs{f - f_n}_\euc^p\, dt \,\Vol_{\g}\\
	& \leq \frac{C}{2\e_n} \int_\M \int_{-\e_n}^{\e_n} \brk{\Abs{ F - F_n}_\euc^p
		+|t(\hn_{f_n}-\hn_f)|_\euc^p} \, dt \,\Vol_{\g}\\
	& \leq \frac{C}{2\e_n} \int_\M \int_{-\e_n}^{\e_n} \brk{\Abs{F - F_n}_\euc^p
		+2^p|t|^p} \, dt \,\Vol_{\g},\\
	\end{split}
	\]
	where we used \eqref{eq:F_ext_f_euc} in the first inequality.
	From \propref{averaged_seq_euc} we conclude that $ \int_\M |f - f_n|_\euc^p \,\Vol_{\g} \to 0  $ as desired.
	
	Next we prove the convergence of $ df_n \to df $ in $ L^p$.
	As before, we extend  $ P_{\parallel}$ to a map $T\tM \to T\tM $ by $ P_{\parallel}(v,w)=(v,0) $. Using Fubini's theorem and \eqref{eq:G to g norm}, we estimate the distance between $df_n$ and $df$ as follows:
	\[ 
	\begin{split}
	&\int_\M  \Abs{df -df_n}_{\g,\euc}^p \,\Vol_{\g}\\
	&\qquad = \frac{1}{2\e_n} \int_\M \int_{-\e_n}^{\e_n} \Abs{\pi^* ( df - df_n)}_{G',\euc}^p\, dt \,\Vol_{\g}\\
	&\qquad = \frac{1}{2\e_n} \int_\M \int_{-\e_n}^{\e_n} \big|\pi^* ( df - df_n) - ( dF- dF_n)\circ P_{\parallel}
	+( dF- dF_n)\circ P_{\parallel}\big|_{G',\euc}^p\, dt \,\Vol_{\g}\\
	&\qquad \leq C (I_1+I_2+I_3),\\
	\end{split}
	\]
	where
	\[ 
	\begin{split}
	I_1&=\frac{1}{2\e_n}\int_\M\int_{-\e_n}^{\e_n} \Abs{( dF- dF_n)\circ P_{\parallel}}_{G',\euc}^p \, dt \,\Vol_{\g}\\
	I_2&=\frac{1}{2\e_n}\int_\M\int_{-\e_n}^{\e_n} \Abs{\pi^*df- dF\circ P_{\parallel}}_{G',\euc}^p \, dt \,\Vol_{\g}\\
	&=\frac{1}{2\e_n}\int_\M\int_{-\e_n}^{\e_n} \Abs{t\, \pi^*(df\circ S_f) }_{G',\euc}^p \, dt \,\Vol_{\g}\\
	I_3&=\frac{1}{2\e_n}\int_\M\int_{-\e_n}^{\e_n} \Abs{dF_n\circ P_{\parallel}-\pi^*  df_n}_{G',\euc}^p \, dt \,\Vol_{\g}\\
	&=\frac{1}{2\e_n}\int_\M\int_{-\e_n}^{\e_n} \Abs{t\, \pi^*(df_n\circ S_{f_n})}_{G',\euc}^p \, dt \,\Vol_{\g}.\\
	\end{split}
	\]
	Note that in $ I_2 $ and $I_3$ we used the explicit formula for $dF$ and $dF_n$ as in \eqref{eq:dF_euc}. 
	
	We evaluate each of these three integrals separately. We start with $ I_1 $. From \propref{averaged_seq_euc} we obtain
	\[ 
	\begin{split}
	I_1&=\frac{1}{2\e_n}\int_\M\int_{-\e_n}^{\e_n} \Abs{( dF- dF_n)\circ P_{\parallel}}_{G',\euc}^p \, dt \,\Vol_{\g}\\
	&\leq \frac{1}{2\e_n}\int_\M\int_{-\e_n}^{\e_n}
	\Abs{ dF - dF_n}_{G',\euc}^p \, dt \,\Vol_{\g} \to 0.
	\end{split}
	\]
	
	Next, we evaluate $I_3$.
	\[ 
	\begin{split}
	I_3&=\frac{1}{2\e_n}\int_\M\int_{-\e_n}^{\e_n} \Abs{t\, \pi^*(df_n\circ S_{f_n})}_{G',\euc}^p \, dt \,\Vol_{\g}\\
	&\leq \frac{C}{2\e_n}\int_{-\e_n}^{\e_n}|t|^p\int_\M  \Abs{df_n\circ S_{f_n}}_{\g,\euc}^p \circ \pi  \,\Vol_{\g}\, dt\\
	&\leq \frac{C}{2\e_n}\int_{-\e_n}^{\e_n}|t|^p\, dt \to 0,
	\end{split}
	\]
	where in the passage to the second line we used \eqref{eq:G to g norm} and Fubini's theorem; in the passage to the last line we used \eqref{eq:bounds on df,dfs_f}.  
	
	$I_2$ can be evaluated in the exact same manner. To conclude
	\[ 
	\int_\M |df - df_n|_{\g,\euc}^p \,\Vol_{\g}\leq C(I_1 + I_2 +I_3 ) \to 0,
	\]
	which completes the proof that $f_n\to f$ in $W^{1,p}(\M;\Rd)$.
	%%%%%%%%%%%%%

	%%%%%
	\bigskip
	\noindent\emph{Step~VI: $ \hn_{f_n} \to \hn_{f} $ in $ W^{1,p}(\M;\Rd) $.}
	\bigskip
	
	As in the previous step, this step is simpler compared to its analog in the proof of Theorem~\ref{THM:MAIN_CONST_CURV}, since we do not need to use the embedding $\iota$. Furthermore, in this case $\nabla \hn_f = d\hn_f$, so there is no need to use the result of Appendix~\ref{sec:appD}.
	
	We start by proving that  $ \hn_{f_n} \to \hn_{f} $ in $ L^p $. From \eqref{eq:dF_euc} and \eqref{eq:F_ext_f_euc} we obtain that
	\[ 
	\begin{split}
	|(\hn_f-\hn_{f_n})\circ \pi|_{\euc}^{p}&=|dF(\partial_t)-dF_n(\partial_t)|_{\euc}^{p}\\
	&\leq |dF-dF_n|_{G',\euc}^{p}|\partial_t|_{G'}^{p}\\
	&=|dF-dF_n|_{G',\euc}^{p}\, ,
	\end{split}
	\] 
	where in the passage to the third line we used the fact that $ |\partial_t|_{G'}=1 $. Thus, 
	\[ 
	\begin{split}
	\int_{\M} |\hn_f-\hn_{f_n}|_{\euc}^{p} \, \Vol_{\g}
	&=\frac{1}{2\e_n}\int_\M\int_{-\e_n}^{\e_n} |(\hn_f-\hn_{f_n})\circ \pi|_{\euc}^{p}  \, dt \,\Vol_{\g}\\
	&\leq  \frac{1}{2\e_n}\int_\M\int_{-\e_n}^{\e_n} |dF-dF_n|_{G',\euc}^{p}  \, dt \,\Vol_{\g}.
	\end{split}
	\]
	From \propref{averaged_seq_euc} we conclude that $ \int_{\M} |\hn_f-\hn_{f_n}|_{\euc}^{p} \, \Vol_{\g} \to 0 $ as desired.
	
	Next we prove that  $ d\hn_{f_n} \to d\hn_{f} $ in $ L^p $. By the Weingarten equation, $ df_n\circ S_{f_n}=-d\hn_{f_n}  $; since $ S=S_f $ it follows that $ df \circ S=-d\hn_f $.
	
	Therefore,
	\[ 
	\begin{split}
	\int_{\M} |d\hn_f-d\hn_{f_n}|_{\g,\euc}^{p} \, \Vol_{\g}
	&=\int_{\M} |df_n \circ S_{f_n}-df \circ S|_{\g,\euc}^{p} \, \Vol_{\g}\\
	&\leq C \int_{\M} \brk{ |df_n \circ S_{f_n}-df_n \circ S|_{\g,\euc}^{p}
		+ |df_n \circ S- df \circ S|_{\g,\euc}^{p}}  \, \Vol_{\g}\\
	&\leq C \int_{\M} \brk{ |df_n \circ (S_{f_n}-S)|_{\g,\euc}^{p}
		+ |df_n- df|_{\g,\euc}^{p}|S|_{\g,\g}^p}  \, \Vol_{\g}\\
	&\leq C \brk{\E_p(f_n) + \int_{\M}|df_n- df|_{\g,\euc}^{p}  \, \Vol_{\g}} \, ,\\
	\end{split}
	\]
	where in the passage to the last line we used the definition of $ \E_p $; we also used the smoothness of $ S $ and the compactness of $ \M $ to obtain a global bound on $ |S|_{\g,\g}^{p} $. Finally, since $ \E_p(f_n) \to 0 $ and $ \int_{\M}|df_n- df|_{\g,\euc}^{p}  \, \Vol_{\g} \to 0 $, we conclude that
	\[ 
	\int_{\M} |d\hn_f-d\hn_{f_n}|_{\g,\euc}^{p} \, \Vol_{\g} \to 0 \, ,
	\]
	which completes the proof that $ \hn_{f_n} \to \hn_{f} $ in $ W^{1,p}(\M;\Rd) $, and thus the proof of \thmref{thm:main_euc}.
	
	%%%%%%%%%%%%%%%%%%%%%%%%%%%%%%%%%%%%%%%%%%%%%%%%%
	\section{Discussion}
	\label{sec:Discussion}
	
	In this section we discuss some aspects and possible generalizations of the main results. 
	
	\paragraph{Constant sectional curvature}
	\thmref{THM:MAIN_CONST_CURV} assumes that $(\N,\h)$ has constant sectional curvature, an assumption that does not appear in the co-dimension zero parallel, \cite[Thm.~3]{KMS19}.
	We do not know whether \thmref{THM:MAIN_CONST_CURV} holds without this assumption (we conjecture that it does); we now highlight its role in the proof.
	
	The main use of the constant curvature assumption is in the construction of the metric $ G $ on $ \tM $. In step~II we proved that for any isometric immersion $ f:(\M,\g) \to (\N,\h) $ with second fundamental form $ \b $, $ G=F^*\h $, where $ F $ is the extension \eqref{eq:F_def} of $f$. Therefore $ G $ is uniquely determined by $ \g $ and $ \b $, as is evident in \eqref{eq:G}. If, however, $ (\N,\h) $ is not of constant sectional curvature, this is no longer true, and as a result we cannot choose $G= F^*\h$ without a priori knowing what is the limit isometric immersion.
	
	The fundamental obstacle when the sectional curvature is not constant is that any tubular neighborhood of a submanifold contains information which is not encoded in the fundamental forms of the submanifold. 
	In this case there is no metric $G$ on $\tM$ such that $\E_p(f_n)\to 0$  implies that the differentials of the normal extensions $F_n$ of $f_n$ tend to $\SO(G,\h)$.
	
	This obstacle might be overcome by replacing the fixed tubular neighborhood by a sequence of shrinking tubular neighborhoods. 
	However, such an approach would require a co-dimension zero rigidity result that is  stronger  than \cite[Thm.~3]{KMS19}, namely, a quantitative version.
	Such an extension will be considered in future works.
	
	\paragraph{Elastic energy functional}
	We now discuss our choice of the energy functional \eqref{eq:energy_euc}, and compare it to \eqref{eq:phys_energy} which is more common in the physics literature.
	Consider first the two stretching terms, 
	\beq\label{eq:stretchings}
	\dist(df,O(\g,\h)) 
	\Textand |g-f^*\h| .
	\eeq
	It is easily verified that $ \dist(df,O(\g,\h))\leq|g-f^*\h|  $, implying that \thmref{thm:main_euc} and \thmref{THM:MAIN_CONST_CURV} are valid for both choices of the stretching term.
	We opted for $ \dist(df,O(\g,\h)) $ because 
	it naturally appears throughout the proof, particularly in \propref{SO dist bound}.
	
	Next consider the two bending terms, 
	\beq\label{eq:bendings}
	|df\circ(S-S_f) | 
	\Textand 
	|\b-\b_f| .
	\eeq
	Shell models of the form ``stretching plus bending" result from a formal asymptotic expansion of a thin elastic body in co-dimension zero.
	Such an expansion assumes that $df$ is close to being an isometry, in which case
	the two bending terms are equivalent, differing in a higher order term in the thickness parameter (this is also true for the stretching terms).
	Therefore, under the assumption that the strains are uniformly small, both bending terms can be used interchangeably.
	
	However, from the perspective of Theorems~\ref{thm:main_euc} and \ref{THM:MAIN_CONST_CURV}, we cannot assume a priori that $df_n$ are uniformly close to isometries.
	While the energy $\E_p(f)$ (for both variants of $ \E_p(f) $) controls the $L^p$ norm of $df$, it does not provide any control on the pointwise norm of $df^{-1}$ (by $df^{-1}$ we mean the inverse of the restriction of $df$ to its image), since the stretching term allows for in-plane shrinking (that is, without bending) of an open neighborhood to a point  at a finite energetic cost.
	Since the ratio between the different bending terms in \eqref{eq:bendings} depends both on $|df|$ and on $|df^{-1}|$, the two bending terms are not uniformly equivalent for the minimizing sequence $f_n$, unless we assume additional control on $ df_n^{-1} $.
	
	Therefore, with either of the two stretching terms in \eqref{eq:stretchings}, our method of proof only works with the $ |df\circ(S-S_f) |$ bending term (note that the term $ df \circ S_f $  naturally appears throughout the proof, as is evident in \propref{dF} and \propref{SO dist bound}).
	
	If one replaces the stretching term with a more physical one that diverges fast enough for singular contractions, so that $\E_p(f_n)\to 0$ gives strong enough control of $|df_n^{-1}|$, then the proof can be carried out with the bending term $ |\b-\b_f| $; yet, depending on the bounds on $|df_n^{-1}|$, there might be a degradation in the exponent $q$ in the $W^{1,q}$ space for which $f_n\to f$.

	\paragraph{Spaces of Sobolev immersions}
	The space $\Imm_p(\M;\N)$ is the natural space on which the functional $\E_p$ is defined.
	In a sense, it is a co-dimension one analogue of the space $\{ f\in W^{1,p}(\Omega;\R^d) ~:~ \det f > 0 \,\text{ a.e.} \}$, however, unlike this space, $\Imm_p(\M;\N)$ has not been thoroughly studied (as far as we know), and there are several outstanding questions about it, even in the simpler case of $\N=\R^{d+1}$.
	In particular, we are not familiar with results regarding the density of smooth immersions in $\Imm_p(\M;\N)$, as well as improved regularity properties (e.g., the range of $p$ for which $\Imm_p(\M;\N)$ consists of continuous functions).

	%%%%%%%%%%%%%%%%%%%%%%%%%%%%%%%%%%%%%%%%%%%%%%%%%
	\appendix
	%%%%%%%%%%%%%%%%%%%%%%%%%%%%%%%%%%%%%%%%%%%%%%%%%
	\section{Proof that $ \b_\M=\b $}
	\label{sec:appA}
	
	We prove that $ \b $ is the second fundamental form of $ \M \simeq \M \times \{0\} $ in $(\tM,G)$.
	Since, by \eqref{eq:G},  $\partial_t$ is normal to $\M$, then $\b_\M$, the second fundamental form  of $\M$ in $\tM$, is given by
	\[
	\b_\M(X,Y) = (S_\M(X),Y)_{\g}, 
	\]
	where
	\[
	S_\M(X) = -P_\parallel(\nabla^G_X \partial_t), 
	\qquad X\in \Gamma(T\M)
	\]
	is the shape operator and $P_\parallel$ is the tangential projection $T\tM|_\M\to T\M$. 
	
	Let $\{x^1,\dots,x^{d+1}\}$ be local coordinates for $\tM$, where $x^{d+1} = t$.
	For $ 1\leq i \leq d $
	\[ 
	S_\M(\partial_i) = 
	-P_\parallel(\nabla_{\partial_i}^G\partial_t) = 
	-\sum_{k=1}^d \left.\Gamma_{i,{d+1}}^{k}\partial_k\right|_{t=0} \, .
	\]
	For $i,k=1,\dots,d$, the Christoffel symbols $\Gamma_{i,{d+1}}^{k}$ at $t=0$ are
	\[ 
	\begin{split}
	\Gamma_{i,{d+1}}^{k}
	&=\dfrac{1}{2}G^{kl}(\partial_iG_{l,d+1} + \partial_{d+1} G_{il} - \partial_l G_{d+1,i})\\
	&=\dfrac{1}{2}G^{kl}(0 + \partial_{t}(u_1^2(t) \pi^*\g -2u_1(t)u_2(t) \pi^*\b +u_2^2(t) \pi^*S^*\g +dt^2)_{il} - 0)\\
	&= - \g^{kl} \b_{il},
	\end{split}
	\]
	where is the passage to the last line we used that $ u_1(0)=1, u_2(0)=0 $ and $ \dot{u}_1(0)=0, \dot{u}_2(0)=1 $.
	It follows that
	\[
	(S_\M(\partial_i),\partial_k)_{\g} = \b_{ik},
	\]
	i.e., $\b_\M = \b$.
	
	%%%%%%%%%%%%%%%%%%%%%%%%%%%%%%%%%%%%%%%%%%%%%%%%%
	\section{Proof that $ A_f $ is an isometry}
	\label{sec:appB}
	
	Recall from \eqref{eq:A_f} that for $ (p,t)\in \tM $ and $ (v,w)\in T_{(p,t)}\tM \simeq T_p\M \times \R $ 
	\[ 
	A_f(v,w)=P_{\gamma^f_p,t}\Brk{u_1(t)O(df_p)(v)-u_2(t)O(df_p)\circ S(v)+w\hn_f(p)}.
	\]
	For $ (v,w),(\xi,\eta)\in T_{(p,t)}\tM \simeq T_p\M \times \R $,
	\[ 
	\begin{split}
	&(A_f(v,w),A_f(\xi,\eta))_{\h}\\
	&\qquad =
	\brk{P_{\gamma^f_p,t}\Brk{u_1(t)O(df_p)(v)-u_2(t)O(df_p)\circ S(v)+w\hn_f(p)}\, , \,
		P_{\gamma^f_p,t}\Brk{u_1(t)O(df_p)(\xi)-u_2(t)O(df_p)\circ S(\xi)+\eta\hn_f(p)}}_{\h}\\
	&\qquad=\brk{u_1(t)O(df_p)(v)-u_2(t)O(df_p)\circ S(v)+w\hn_f(p)\, , \,
		u_1(t)O(df_p)(\xi)-u_2(t)O(df_p)\circ S(\xi)+\eta\hn_f(p)}_{\h}\\
	&\qquad=u_1^2(t)\brk{O(df_p)(v),O(df_p)(\xi)}_{\h}\\
	&\qquad \qquad -u_1(t)u_2(t) \Brk{\brk{O(df_p)(v),O(df_p)\circ S(\xi)}_{\h}+\brk{O(df_p)\circ S(v),O(df_p)(\xi)}_{\h}}\\
	&\qquad \qquad + u_2^2(t) \brk{O(df_p)\circ S(v),O(df_p)\circ S(\xi)}_{\h} + w\eta\brk{\hn_f(p),\hn_f(p)}_{\h}\\
	&\qquad=u_1^2(t)\brk{v,\xi}_{\g} -u_1(t)u_2(t)\Brk{\brk{v,S(\xi)}_{\g}+\brk{S(v),\xi}_{\g}} + u_2^2(t)\brk{S(v),S(\xi)}_{\g} +w\eta\\
	&\qquad=u_1^2(t)\brk{v,\xi}_{\g} -u_1(t)u_2(t)\Brk{\brk{S(v),\xi}_{\g}+\brk{S(v),\xi}_{\g}} + u_2^2(t)\brk{S(v),S(\xi)}_{\g} + w\eta\\
	&\qquad=u_1^2(t)\pi^*\g\brk{(v,w),(\xi,\eta)}-2u_1(t)u_2(t)\pi^*\b\brk{(v,w),(\xi,\eta)} + u_2^2(t)\pi^*S^*\g\brk{(v,w),(\xi,\eta)} + dt^2\brk{(v,w),(\xi,\eta)}\\
	&\qquad=\brk{(v,w),(\xi,\eta)}_G.
	\end{split}
	\]
	In the passage to the second equality we used that parallel transport is a linear isometry;
	in the third equality we used that $ \hn_f(p) $ is orthogonal to the image of $ O(df_p) $;
	in the fourth equality we used that $ O(df_p):T_p\M \to T_{f(p)}\N $ is an orthogonal map;
	and in the fifth equality we used that $ S $ is symmetric.

	%%%%%%%%%%%%%%%%%%%%%%%%%%%%%%%%%%%%%%%%%%%%%%%%%
	\section{Proof of \lemref{par_aprox}}
	\label{sec:appC}
	
	%%%%%%%%%%%%%
	We cover $ (\N,\h) $ with a finite number of precompact coordinate neighborhoods. 
	From the Lebesgue's number lemma it follows that we could have chosen $ \e $  small enough such that every unit speed geodesic $ \gamma:[-\e,\e]\to \N $ is contained in a single coordinate neighborhood. 
	We denote $ X(t):=P_{\gamma,t}(x) $, i.e., $ X(t) $ is the parallel transport of $ x $ along $ \gamma $.
	In coordinates, $ X(t) $ is the solution to the following system of linear differential equations:
	\[ 
	\dot{X}^k+\Gamma_{ij}^k(\gamma) \dot{\gamma}^iX^j=0,
	\]
	where $\Gamma_{ij}^k$ are the Christoffel symbols of $\h$.
	From the precompactness of the coordinate neighborhood the Christoffel symbols have a uniform bound. 
	Since $ \gamma $ is of unit speed and again from precompactness of the coordinate neighborhood, the coordinate components of the velocity $ \dot{\gamma}^i $ are also uniformly bounded independently of $\gamma$.
	
	This implies that $ |\dot{X^k}|\leq C\sum_j|X^j|\leq C |X|$ for some constant $C$ independent of $\gamma$.
	Therefore we obtain:
	\beq
	\label{eq:X-x_0 bound}
	\begin{split}
		|X(t)-X(0)|&\leq \int_{0}^{t} |\dot{X}(s)|\,ds
		\leq C\int_{0}^{t} |X(s)|\,ds
		\leq C|X(0)||t| +C\int_{0}^{t} |X(s)-X(0)|\,ds.
	\end{split}
	\eeq
	Note that this bound only makes sense in the coordinate neighborhood, since intrinsically $X(0)$ and $X(t)$ are not in the same fiber.
	We now bound $ \int_{0}^{t} |X(s)-X(0)|\,ds $.
	Define $ y(t)=|X(t)-X(0)| $ and $ Y(t)=\int_{0}^{t} y(s)\,ds $. Then $ \dot{Y}(t)=y(t) $ and from the above calculation we get
	\[ 
	\dot{Y}(t)-CY(t)\leq C|X(0)||t|
	\]
	We multiply both sides by $ e^{-Ct} $ and obtain
	\[ 
	\begin{split}
	\frac{d}{dt}(Y(t)e^{-Ct})&\leq C|X(0)||t|e^{-C t}\\
	Y(t)e^{-C t} -\cancel {Y(0)}&\leq C|X(0)|\int_{0}^{t} |s|e^{-C s}ds\\
	Y(t)e^{-C t}&\leq  C|X(0)| \left[-\frac{se^{-Cs}}{C}-\frac{e^{-Cs}}{C^2}\right]_{0}^{t}\\
	Y(t)e^{-C t}&\leq  C|X(0)| \left[-\frac{te^{-Ct}}{C}-\frac{e^{-Ct}}{C^2} + \frac{1}{C^2}\right]\\
	Y(t)&\leq |X(0)| \left[-t -\frac{1}{C} + \frac{e^{Ct}}{C}\right]\\
	Y(t)&\leq |X(0)| \left[\frac{Ce^{Ct}}{2}t^2\right]\\
	Y(t)&\leq |X(0)| \left[\frac{Ce^{C}}{2}|t|\right],\\
	\end{split}
	\]
	where in the last two lines we used Lagrange remainder for the Taylor series of $ e^{Ct} $ and the fact that $ |t|<1 $. 
	In the passage to the third line we assumed that $ t\geq 0 $; the case $t<0$ is similar.
	We obtain that $ \int_{0}^{t} |X(s)-X(0)|\,ds=Y(t)\leq C|X(0)||t| $. Therefore, combined with \eqref{eq:X-x_0 bound} we conclude that
	\[ 
	|X(t)-X(0)|\leq C|X(0)||t|.
	\]
	Since $ \gamma $ is of unit speed and the coordinate neighborhood is precompact we have  that 
	\[ 
	|d\iota_{{\gamma(t)}} - d\iota_{\gamma(0)}|\leq C |t|,
	\]
	where $ C $ does not depend on $ \gamma $ (it is simply obtained by a bound on the hessian of $ \iota $ and the fact that $\gamma$ is of unit speed).
	In addition there exists a uniform bound on $ d\iota $. 
	We can finally estimate $R_{\gamma,x}(t)$; in the estimate below, we again consider $X$ and $d\iota$ as matrices in their coordinates representations,
	\[ 
	\begin{split}
	|R_{\gamma,x}(t)|_\euc&=|d\iota \circ P_{\gamma,t}(x)-d\iota (x)|_\euc\\
	&\leq |d\iota_{\gamma(0)}(X(t))-d\iota_{\gamma(0)} (X(0))|
	+ |(d\iota_{{\gamma(t)}} - d\iota_{\gamma(0)})(X(t))|\\
	&\leq |d\iota_{\gamma(0)}||X(t)-X(0)|
	+|d\iota_{{\gamma(t)}} - d\iota_{\gamma(0)}||X(t)-X(0)|
	+|d\iota_{{\gamma(t)}} - d\iota_{\gamma(0)}||X(0)|\\
	&\leq C |X(0)||t|\\
	&\leq C |x|_{\g}|t|,
	\end{split}
	\]
	where we used the fact that $ |t|<1 $ and that the norm in coordinates is equivalent to $\g$ (again form precompactness). Since $ (\N,\h) $ is covered by a finite number of coordinate neighborhoods these bounds can be taken to be global. 
	\qed
	%%%%%%%%%%%%%
	
	%%%%%%%%%%%%%%%%%%%%%%%%%%%%%%%%%%%%%%%%%%%%%%%%%
	\section{Equivalent definitions for the $ W^{1,p} $ convergence of vector fields along $ W^{1,p} $ maps }
	\label{sec:appD}

	\begin{proposition}\label{pn:appD}
		Let $ (\M,\g) $ and $ (\N,\h) $ be two compact Riemannian manifolds. For $ 1\leq p <\infty $,
		let $ f_n $ and $ f $ be in $ W^{1,p}(\M;\N) $ and let $ V_{f_n} $ and $ V_f $ be vector fields in $ W^{1,p}(\M;T\N) $ covering $ f_n $ and $ f $ respectively, with a uniform pointwise bound (i.e., there exists $R>0$ such that for all $ p\in \M $ and $ n $, $ |V_{f_n}(p)|_{\h}\leq R $).
		
		Then $ V_{f_n} \to V_f $ in $ W^{1,p}(\M;T\N) $ if and only if $ f_n \to f $ in $ W^{1,p}(\M;\N) $ and  for some isometric embedding  $ \iota:(\N,\h) \to \R^D $ 
		\begin{align}\label{eq:vector_conv}
		&d\iota \circ V_{f_n} \to d\iota \circ V_{f} 				&\text{in } & L^p(\M;\R^D)\\
		&d\iota \circ \nabla V_{f_n} \to d\iota \circ \nabla V_{f}	&\text{in } & L^p(\M;T^*\M \otimes \R^D) \nonumber
		\end{align}
		In this case, this convergence holds for every isometric embedding $ \iota:(\N,\h) \to \R^D $.
	\end{proposition}
	
	\begin{proof}
		We first show that the $ L^p{(\M;T\N)} $ convergence of $ V_{f_n} $ is defined independently of the choice of metric on $T\N$.
		Every metric $ \tilde{\h} $ on $ T\N $ induces a distance function $ d_{\tilde{\h}} $ and thus the $ L^p{(\M;T\N)} $ convergence can be defined by:
		\beq
		\label{eq:intrinsic L^p}
		\int_{\M} d^p_{\tilde{\h}}(V_f,V_{f_n}) \Vol_{\g} \to 0.
		\eeq
		We show that this convergence does not depend on the choice of the metric $ \tilde{\h} $. Let $ C\subset T\N $ be a compact submanifold of $ T\N $ such that $ V_{f_n} $ is contained in it. Such a submanifold exists since $ \N $ is compact and $ V_{f_n} $ has a uniform pointwise bound. Let $\tilde{\h}_{C} $ be the induced metric on $ C $ and $ d_{\tilde{h}_{C}} $ the distance function induced by $ \tilde{\h}_{C} $. 
		Since the extrinsic ($ d_{\tilde{\h}} $) and intrinsic ($ d_{\tilde{h}_{C}} $) distances on a compact manifold are equivalent (this follows from a compactness argument, alternatively see \cite{m.kohan}), it follows that changing $ d_{\tilde{\h}} $ to $ d_{\tilde{h}_{C}} $ in \eqref{eq:intrinsic L^p} does not change the $L^p$ convergence.
		Furthermore, if we endow $ T\N $ with a different metric $ \tilde{r} $ then from the compactness of $ C $, $ \tilde{\h}_{C} $ and $ \tilde{r}_{C} $ are equivalent and in turn so are $ d_{\tilde{\h}_{C}} $ and $ d_{\tilde{r}_{C}} $. 
		Therefore, the $ L^p $ convergence of $ V_{f_n} $ as defined in \eqref{eq:intrinsic L^p} does not depend on the choice of metric on $ T\N $. 
		
		Next, we prove that the definition of $ L^p $ convergence by isometrically embedding $ (T\N,\h) $ into $ (\R^K,\euc) $ is equivalent to the definition in \eqref{eq:intrinsic L^p}.  Let $ \tilde{\h} $ be a metric on $ T\N $ and let $ i:(T\N,\tilde{\h}) \to (\R^K,\euc) $ be a smooth isometric embedding into Euclidean space. Since $ i(C) \subset \R^K $ is an embedded compact submanifold, the extrinsic and intrinsic distances on it are equivalent. We obtain that
		\[ 
		\begin{split}
		\int_{\M} d^p_{\tilde{\h}}(V_f,V_{f_n}) \Vol_{\g} \to 0
		&\iff \int_{\M} d^p_{\tilde{\h}_{C}}(V_f,V_{f_n}) \Vol_{\g} \to 0\\
		&\iff \int_{\M} d^p_{\euc_{i(C)}}(V_f,V_{f_n}) \Vol_{\g} \to 0\\	
		&\iff \int_{\M} |i\circ V_f - i\circ V_{f_n}|_{\euc}^p \Vol_{\g} \to 0,
		\end{split}
		\]
		and therefore the two definitions are equivalent. Since the $ L^p $ convergence defined by \eqref{eq:intrinsic L^p}  does not depend on $ \tilde{h} $ so does the $ L^p $ convergence by embedding into euclidean space.
		\bigskip

		\emph{$(\Leftarrow): $ Assume that $ f_n \to f $ in $ W^{1,p}(\M;\N) $, \eqref{eq:vector_conv} holds and prove that $ V_{f_n} \to V_f $ in $ W^{1,p}(\M;T\N) $. }
		First, we prove that $ V_{f_n} \to V_{f}  $ in $ L^p(\M;T\N) $. Let $ \iota: (\N,\h) \to (\R^D,\euc) $ be a smooth isometric embedding (note that this is an embedding of $ \N $ and not $ T\N $). Let $ T\iota: T\N \to \R^D \times \R^D $ be its differential (in the sense of a bundle morphism) defined by $ T\iota (p,v)=(\iota(p),d\iota_p(v)) $. 
		It is easily verified that $ T\iota $ is a smooth immersion, i.e., $ T^2\iota: TT\N \to \R^{4D} $ is injective. 
		If we endow $ T\N $ with the pullback metric $ (T\iota)^*\euc $ then $ T\iota  $ is a smooth isometric embedding. 
		Recall that $ V_f(p) = (f(p),\nu_f(p)) $ where we usually identify $ V_{f}(p) $ with $ \nu_f(p) $ (see \secref{subsec:sobolev_immersions}). 
		We then have
		\[ 
		\begin{split}
		&\int_{\M} |T\iota \circ V_f - T\iota \circ V_{f_n}|_{\euc}^p \Vol_{\g}\\
		&\qquad \qquad=\int_{\M} {\brk{|\iota\circ f - \iota \circ f_n|_{\euc}^2+|d\iota\circ \nu_{f} - d\iota \circ \nu_{f_n}|_{\euc}^2}}^{p/2} \Vol_{\g}\\
		&\qquad \qquad \leq C \int_{\M} |\iota\circ f - \iota \circ f_n|_{\euc}^p+|d\iota\circ \nu_{f} - d\iota \circ \nu_{f_n}|_{\euc}^p \Vol_{\g} \to 0,\\
		\end{split}
		\]
		where in the passage to the second line we used the definition of $ T\iota $ and $ V_f $; the limit is obtained since $ f_n \to f $ in $ W^{1,p}(\M;\N) $ and $ d\iota \circ V_{f_n} \to d\iota \circ V_f $ in $ L^p(\M;\R^D) $ (where we identified $ V_f $ and $ \nu_f $ in this instance). Since the $ L^p(\M;T\N) $ convergence does not depend on the choice of metric on $ T\N $ this convergence is equivalent to the "standard" $ L^p(\M;T\N) $ convergence when $ T\N $ is endowed with the Sasaki metric $ S_{\h} $ \eqref{eq:Sasaki} (i.e., the one defined by isometrically embedding $(T\N,S_{\h})$ into Euclidean space).
		
		Next, we prove the convergence $ dV_{f_n} \to dV_f $ in $ L^p(\M;T^*\M\otimes TT\N ) $, i.e., that for every smooth isometric embedding $ i:(T\N,S_\h) \to (\R^K,\euc) $ we obtain the convergence $ d(i\circ V_{f_n}) \to d(i\circ V_f) $ in $ L^p(\M; T^*\M \otimes R^K) $. From \cite[Def.~4.1, Prop.~4.4]{CVS16} it suffices to prove that $ dV_{f_n} $ converges to $ dV_f $ in measure and that the sequence of norms $ \brk{|dV_{f_n}|_{\g,S_\h}} $ converges to $ |dV_{f}|_{\g,S_\h} $ in $ L^p(\M;\R) $.
		
		\noindent We start by proving the $ L^p $ convergence of the norms. 
		\[ 
		\begin{split}
		&\int_{\M} \left| |dV_{f}|_{\g,S_\h}-|dV_{f_n}|_{\g,S_\h} \right|^p \Vol_{\g}\\
		&\qquad \qquad=\int_{\M} \left| \sqrt{|df|_{\g,\h}^2+|\nabla V_{f}|_{\g,\h}^2}-\sqrt{|df_n|_{\g,\h}^2+|\nabla V_{f_n}|_{\g,\h}^2} \right|^p \Vol_{\g}\\
		&\qquad \qquad=\int_{\M} \left|\frac{ |df|_{\g,\h}^2+|\nabla V_{f}|_{\g,\h}^2-|df_n|_{\g,\h}^2-|\nabla V_{f_n}|_{\g,\h}^2}
		{\sqrt{|df|_{\g,\h}^2+|\nabla V_{f}|_{\g,\h}^2}+\sqrt{|df_n|_{\g,\h}^2+|\nabla V_{f_n}|_{\g,\h}^2}} \right|^p \Vol_{\g}\\
		&\qquad \qquad\leq C \int_{\M} \left|\frac{ |df|_{\g,\h}^2-|df_n|_{\g,\h}^2}{|df|_{\g,\h}+|df_n|_{\g,\h}}\right|^p
		+\left|\frac{ |\nabla V_{f}|_{\g,\h}^2-|\nabla V_{f_n}|_{\g,\h}^2}{|\nabla V_{f}|_{\g,\h}+|\nabla V_{f_n}|_{\g,\h}} \right|^p \Vol_{\g}\\
		&\qquad \qquad= C \int_{\M} \left| |d\iota \circ df|_{\g,\euc}-|d\iota\circ df_n|_{\g,\euc}\right|^p
		+\left| |d\iota \circ\nabla V_{f}|_{\g,\euc}-|d\iota \circ\nabla V_{f_n}|_{\g,\euc} \right|^p \Vol_{\g}
		\to 0,\\
		\end{split}
		\]
		where in the passage to the second line we used the definition of the Sasaki metric (with relation to the metric $ \h $ and the Levi-Civita connection) and that $ d\pi \circ d\hn_f=df $; in the passage to the fifth line we used that $ \iota:(\N,\h) \to (\R^D,\euc) $ is an isometric immersion. Finally, the limit is obtained since $ f_n \to f $ in $ W^{1,p}(\M;\N) $ and $ d\iota \circ \nabla V_{f_n} \to d\iota \circ \nabla V_{f} $ in $  L^p(\M;T^*\M \otimes \R^D) $.

		\noindent We now prove that $ dV_{f_n} $ converges to $ dV_f $ in measure. From \cite[Prop.~4.5]{CVS16} it suffices to prove that $ V_{f_n} $ converges to $ V_f $ in measure and that $ \brk{|dV_{f_n}|_{\g,S_\h}} $ converges to $ |dV_{f}|_{\g,S_\h} $ in $ L^1(\M;\R) $.
		These requirements are satisfied immediately. First, the $ L^p $ convergence of $ V_{f_n}  $ to  $ V_f $ implies convergence in measure. Second, since $ \M $ is compact the $ L^p $ convergence of the norms implies $ L^1 $ convergence.
		
		\bigskip
		
		\emph{$ (\Rightarrow): $ Assume that $ V_{f_n} \to V_f $ in $ W^{1,p}(\M;T\N) $ and prove that $ f_n \to f $ in $ W^{1,p}(\M;\N) $ and \eqref{eq:vector_conv} holds.}
		Let $ \pi:T\N \to \N $ be the projection, then $ f_n=\pi \circ V_{f_n} $. From the definition of the Sasaki metric \eqref{eq:Sasaki} and the fact that $ V_{f_n} \to V_f $ in $ W^{1,p}(\M;T\N) $ we obtain that $ f_n \to f $ in $ W^{1,p}(\M;\N) $.
		$ V_{f_n} \to V_f $ in $ W^{1,p}(\M;T\N) $ implies that $ V_{f_n} \to V_f $ in $ L^p(\M;\N) $, and as shown above, this convergence does not depend on the choice of metric on $ T\N $.  
		Let $ \iota: (\N,\h) \to (\R^D,\euc) $ be a smooth isometric embedding. As shown above, if we endow $ T\N $ with the pull back metric $( T\iota)^*\euc $, then $ T\iota $ is a smooth isometric embedding. Therefore, $ \int_{\M} |T\iota \circ V_f - T\iota \circ V_{f_n}|_{\euc}^p \,\Vol_{\g} \to 0 $.  We then get
		\[ 
		\int_{\M} |d\iota\circ \nu_{f} - d\iota \circ \nu_{f_n}|_{\euc}^p\,\Vol_{\g}
		\leq \int_{\M} |T\iota \circ V_f - T\iota \circ V_{f_n}|_{\euc}^p \,\Vol_{\g} \to 0,
		\]
		where the inequality follows from the fact that one of the components of $T\iota \circ V_f$ is $d\iota\circ \nu_{f}$.
		Under the identification of $ V_{f_n} $ with $ \nu_{f_n} $ we obtain that $ d\iota \circ V_{f_n} \to d\iota \circ V_{f} $  in $ L^p(\M;\R^D) $.
		
		$ V_{f_n} \to V_f $ in $ W^{1,p}(\M,T\N) $ implies that $ dV_{f_n} \to dV_f $ in measure \cite[Def.~4.1, Prop.~4.4]{CVS16}. This in turn implies that every subsequence of $ dV_{f_n} $ has a subsequence that converges a.e.\ to $ dV_f $. 
		Since the connector $ K:T\N \to TT\N $ is continuous, $ d\iota:T\N \to R^D $ is continuous and $ \nabla V_{f_n}=K\circ dV_{f_n}  $ we obtain that every subsequence of $ d\iota\circ \nabla V_{f_n} $ has a subsequence that converges a.e.\ to $ d\iota \circ \nabla V_f $. From \cite[Def.~4.1]{CVS16} we know that $ V_{f_n} \to V_f $ in $ W^{1,p}(\M,T\N) $ implies that the sequence of norms $ \brk{|dV_{f_n}|_{\g,S_\h}} $ converges to $ |dV_{f}|_{\g,S_\h} $ in $ L^p(\M;\R) $. 
		From the definition of the Sasaki metric \eqref{eq:Sasaki} and the fact that $ \iota:(\M,\g) \to (\R^D,\euc) $ is an isometric embedding we obtain that the sequence of norms $ \brk{|d\iota \circ \nabla V_{f_n}|_{\g,\euc}} $ converges to $ |d\iota \circ \nabla V_{f}|_{\g,\euc} $ in $ L^p(\M;\R) $.
		To conclude, every subsequence of $ d\iota\circ \nabla V_{f_n} $ has a subsequence that converges a.e.\ to $ d\iota \circ \nabla V_f $, and $ \lim_{n\to \infty}\int_{\M} |d\iota\circ \nabla V_{f_n}|_{\g,\euc}^p\Vol_{\g}=\int_{\M} |d\iota\circ \nabla V_{f}|_{\g,\euc}^p\Vol_{\g}   $.
		From \cite[prop.4.2.6]{willem13} we obtain that the subsequence (of the subsequence) converges to $ d\iota\circ \nabla V_{f} $ in $ L^p(\M;\R^D) $. 
		Finally, since every subsequence of $ d\iota\circ \nabla V_{f_n} $ has a subsequence that converges to $ d\iota \circ \nabla V_f $ in $ L^p(\M;\R^D) $ we get that
		$ d\iota \circ \nabla V_{f_n} \to d\iota \circ \nabla V_{f}$ in $ L^p(\M;T^*\M \otimes \R^D)$,
		which completes the proof. 
%		\qed
	\end{proof}
	%%%%%%%%%%%%%
	%%%%%%%%%%%%%
	\footnotesize{
	\bibliographystyle{amsalpha}
	\newcommand{\etalchar}[1]{$^{#1}$}
	\providecommand{\bysame}{\leavevmode\hbox to3em{\hrulefill}\thinspace}
	\providecommand{\MR}{\relax\ifhmode\unskip\space\fi MR }
	\providecommand{\MRhref}[2]{%
		\href{http://www.ams.org/mathscinet-getitem?mr=#1}{#2}
	}
	\providecommand{\href}[2]{#2}
	}
	
\end{document}